\documentclass[reqno,11pt]{amsart}
\usepackage{amsmath,amssymb,latexsym,soul,cite,mathrsfs}
\usepackage{color,enumitem,graphicx}
\usepackage[colorlinks=true,urlcolor=blue,
citecolor=red,linkcolor=blue,linktocpage,pdfpagelabels,
bookmarksnumbered,bookmarksopen]{hyperref}
\usepackage[english]{babel}
\usepackage[left=2.6cm,right=2.6cm,top=2.9cm,bottom=2.9cm]{geometry}
\usepackage[hyperpageref]{backref}

\usepackage[pagewise]{lineno}
\usepackage{lineno}

\numberwithin{equation}{section}
\newtheorem{theorem}{Theorem}[section]
\newtheorem{lemma}[theorem]{Lemma}
\newtheorem{corollary}[theorem]{Corollary}
\newtheorem{proposition}[theorem]{Proposition}

\renewcommand{\epsilon}{\varepsilon}

\renewcommand{\rightarrow}{\to}

\newcommand{\ud}{\mathrm{d}}

\title[]{On a  supercritical k-Hessian inequality of Trudinger-Moser type and extremal functions}

\author[J.F.\ de Oliveira]{Jos\'{e} Francisco de Oliveira}\thanks{First author was partially supported by  National Council for Scientific and Technological Development (CNPq) \# 309491/2021-5}
\author[J.M. do \'{O}]{Jo\~{a}o Marcos do \'{O}}
\thanks{The second author was partially supported by Para\'iba State Research Foundation (FAPESQ) \#3034/2021, and National Council for Scientific and Technological Development (CNPq) \#312340/2021-4 and \#429285/2016-7.}
\author[P. Ubilla]{Pedro Ubilla}\thanks{The third author was partially supported by FONDECYT \#1220675.}

\address[J.F.\ de Oliveira]{
\newline\indent Department of Mathematics
	\newline\indent 
	Federal University of Piau\'{i}
	\newline\indent
	64049-550 Teresina, PI, Brazil}
	\email{\href{mailto:jfoliveira@ufpi.edu.br}{jfoliveira@ufpi.edu.br}}	
\address[J.M. do \'{O}]{\newline\indent Department of Mathematics
\newline\indent 
Federal University of Para\'{\i}ba
\newline\indent
58051-900 Jo\~{a}o Pessoa, PB, Brazil}
\email{\href{mailto:jmbo@pq.cnpq.br}{jmbo@pq.cnpq.br}}
\address[P. Ubilla]{\newline\indent Departamento de Matematica
\newline\indent 
Universidad de Santiago de Chile
\newline\indent
Casilla 307, Correo 2, Santiago, Chile}
\email{\href{mailto:pedro.ubilla@usach.cl}{pedro.ubilla@usach.cl}}
\begin{document}

\maketitle

\begin{abstract}
We establish a supercritical Trudinger-Moser type inequality for the $k$-Hessian operator on the space of the $k$-admissible radially symmetric functions $\Phi^{k}_{0,\mathrm{rad}}(B)$, where $B$ is the unit ball in $\mathbb{R}^{N}$. We also prove the existence of extremal functions for this new supercritical inequality. 
\end{abstract}
\section{Introduction}
The critical  Sobolev inequality tells us that the Sobolev space  $W^{1,p}_{0}(\Omega)$  is continuously embedded in $L^{p^*}(\Omega)$, with $p^*=Np/(N-p)$ and $p<N$, where $\Omega\subset\mathbb{R}^N (N\ge 2)$  is a smooth bounded domain.  This embedding is optimal in the sense that the maximal growth for a function $g:\mathbb{R}\to \mathbb{R}$ such that $g\circ u\in L^{1}(\Omega)$ whenever $u\in W^{1,p}_{0}(\Omega)$ is determined by $g(t)=|t|^{p^*}$. In particular, we can not replace the exponent $p^*$ by any $q>p^*$, which causes  a loss of compactness and affects the existence of an extremal function for the associated maximal problem. 
Let $B$ be the unit ball in $\mathbb{R}^{N}$ and $W^{1,2}_{0,\mathrm{rad}}(B)$ the Sobolev space of radially symmetric functions about the origin. The authors in \cite{MR3514752} were able to prove a Sobolev-type inequality giving an embedding into non-rearrangement invariant spaces $L_{2^*+|x|^{\alpha}}(B)$, the variable exponent Lebesgue space, which goes beyond the critical exponent $2^{*}$. 
Further, they analyzed the attainability of the associated supercritical extremal problem. Precisely, the supremum 
\begin{equation}\label{primary}
\mathcal{U}_{N,\alpha}=\sup\left\{\int_{B}|u|^{2^{*}+|x|^{\alpha}}\mathrm{d} x\; |\; u\in W^{1,2}_{0,\mathrm{rad}}(B),\; \|\nabla u\|_{L^2(B)}=1 \right\},
\end{equation}
is finite and is
attained for some $u\in W^{1,2}_{0,\mathrm{rad}}(B)$  provided that 
\begin{equation}\label{alphaDoRUUb}
 0<\alpha<\min\left\{N/2, N-2\right\}.
 \end{equation}
It is worth pointing out that \eqref{primary} is closely  related to the  following supercritical elliptic equation for the Laplace operator
\begin{equation}\label{primary-problem}
\left\{
\begin{aligned}
  -\Delta u & =  u^{2^{*}+|x|^{\alpha}-1}, \; u >0 & \;\;\mbox{in}\;\;& B\\
  u & =  0 & \;\;\mbox{on}\;\;&\partial B.
\end{aligned}
\right. 
\end{equation}
The existence of solutions for \eqref{primary-problem} is also established in \cite{MR3514752} under condition \eqref{alphaDoRUUb}.
In the recent paper \cite{ORU}, a critical Moser-type inequality with loss of compactness due to infinitesimal shocks is obtained.  For more extensions and related results, we recommend \cite{Ngu,CLDOUN, JDE2019,Cao, OLUBNA} and references therein.

Let us consider the $k$-Hessian operator $F_k$,  $1\le k\le N$, 
\begin{equation}\nonumber
F_k[u]=\sum_{1\le i_1<\dots<i_k\le N}\lambda_{i_1}\dots\lambda_{i_k},
\end{equation}
where $\lambda=(\lambda_1, \dots, \lambda_N)$ are the eigenvalues of the real symmetric Hessian matrix $D^{2}u$ of a function $u\in C^{2}(\Omega)$. Note that 
$F_1[u]=\Delta u$ and $F_{N}[u]=\det(D^2u)$ is the Monge-Amp\`{e}re operator. 

   In the well-known paper \cite{MR3487276},
  in order not to lose the ellipticity property of the operators  $F_k$, with $k>1$, L. Caffarelli et al.  proposed the space  of the $k$-admissible functions $\Phi^{k}_0(\Omega),$ that is, the subspace of the functions $u\in C^{2}(\Omega)$ vanishing on $\partial\Omega$ such that $F_{j}[u]\ge 0,\; j=1,\dots,k$.  To state our results properly,  we recall  some properties of $\Phi^{k}_0(\Omega)$ (see  \cite{MR3487276,Wang2,MR1368245, TianWang,ChouWang} for more details).  In \cite{Wang2}, it was observed that the expression
\begin{equation}\label{admissible-norm}
\|u\|_{\Phi_0^k}= \left(\int_{\Omega}(-u)F_{k}[u]\ud x\right)^{\frac{1}{k+1}}, \;\; u\in \Phi^{k}_{0}(\Omega
)
\end{equation}
defines a norm on  $\Phi^{k}_{0}(\Omega
)$. In addition,  in the {\it Sobolev case}  $2k < N,$  one has  the following  Sobolev-type inequality 
\begin{equation}\label{S-inequality}
\|u\|_{L^{p}(\Omega)} \le C \, \|u\|_{\Phi_0^k},  \quad \text{for all} \quad p \in [1,k^*],
\end{equation}
where the critical exponent $k^*$ is defined by
$$
k^* = \frac{N(k+1)}{N-2k}.
$$ 
Moreover, in the {\it Sobolev limit case}, $2k = N,$   it holds 
\begin{equation}
\|u\|_{L^{p}(\Omega)} \le C \, \|u\|_{\Phi_0^k},  \quad \text{for all} \quad p \in [1,\infty).
\end{equation}
In  \cite{OLUBNA}, for the  {\it Sobolev case}  $2k < N,$ it was proved a supercritical  Sobolev type inequality for $k$-admissible radially symmetric functions space $\Phi^{k}_{0,\mathrm{rad}}(B)$. Precisely, for $\alpha>0,$ the supremum 
\begin{equation}\label{mainS}
\mathcal{U}_{k, N,\alpha}=\sup \left\{\int_{B}|u(x)|^{k^{*}+|x|^{\alpha}}\ud x\;\;|\;\; u\in\Phi^{k}_{0,\mathrm{rad}}(B)\;,\; \|u\|_{\Phi^{k}_{0}}=1\right\}
\end{equation}
is finite. Furthermore,  $\mathcal{U}_{k, N,\alpha}$ is attained provided that
$$0<\alpha
<\frac{N-2k}{k}.$$ 

For the Sobolev limit case, $2k=N$, in 
 \cite{TianWang} it was established the following sharp Trudinger-Moser type inequality
\begin{equation}\label{TMinequality} 
\sup_{\|u\|_{\Phi^{k}_{0}}=1} \int_{\Omega}\mathrm{e}^{\mu_N|u|^{\frac{N+2}{N}}}\ud x\le C
\end{equation}
with the critical constant
\begin{eqnarray}
\mu_N &= N\left[\frac{\omega_{N-1}}{k} {\binom{N-1}{k-1}}\right]^{2/N},
\end{eqnarray}
where $C$ is a positive constant depending only on $N$ and $\mathrm{diam}(\Omega)$.  

We observe that, as was proved in \cite[Remark~3.1]{TianWang}, the exponent $\beta_0=(N+2)/N$ is optimal in the sense that we can not replace it in \eqref{TMinequality} by any $\beta_0+c,$ with $c>0$. 
In \cite{Hempel,Moser} it was established the same conclusion for the classical Sobolev space $W^{1,N}_{0}(\Omega)$.

The main target of this work is to analyze \eqref{TMinequality} for the supercritical growth range. Hence, throughout the rest of this paper, we assume $N\ge2, \; 2k=N$ and $f:[0,1)\rightarrow[0,\infty)$ is a continuous function 
such that
\begin{equation}\nonumber
    f(0)=0 \;\;\mbox{and}\;\; f>0 \;\; \mbox{on}\;\; r\in (0,1).
\end{equation}
Our first result reads below.
\begin{theorem}\label{thm1} Assume that there exists $\sigma>1$ such that
\begin{equation}\label{limsup}
   \limsup_{r \rightarrow 0} f(r)|\ln r|^{\sigma} < \infty.
\end{equation}
Then,
\begin{equation}\label{Sad2}
\mathrm{MT}_{N}(f)=\sup_{u\in \Phi^{k}_{0,\mathrm{rad}}(B),\;\; \|u\|_{\Phi^{k}_{0}}=1}\int_{B}\mathrm{e}^{\mu_N|u|^{\frac{N+2}{N}+f(|x|)}}\ud x< \infty,
\end{equation}
where $\mu_N=N\big[\frac{\omega_{N-1}}{k}\binom{N-1}{k-1}\big]^{\frac{2}{N}}.$ 
\end{theorem}
The above theorem can be seen as a perturbation of the critical inequality proved in \cite{TianWang}, which means that we must study functional inequalities in the supercritical regime. We emphasize that we do not impose any condition on $f(r)$ for $r$ near $1,$ which allows $f(r)$ to be (possibly) unbounded near $1$. 

We note that supercritical Trudinger-Moser  inequalities such as in \eqref{Sad2} were investigated recently in \cite{Ngu2} for the $N$-Laplace operator in the classical Sobolev space  $W^{1,N}_{0,\mathrm{rad}}(B)$, which includes  the  linear case $F_1[u]=\Delta u$ on $W^{1,2}_{0,\mathrm{rad}}(B)$.  Our results improve and  complement relevant results in  \cite{Ngu2} for the $k$-admissible functions space $\Phi^{k}_{0,\mathrm{rad}}(B)$ and  fully nonlinear $k$-Hessian range $k>1$.

Another interesting and intriguing question is whether Trudinger–Moser type inequalities have maximizers. In this sense,  it was investigated in \cite{MZ2020} the existence of a extremal function for the $k$-Hessian inequality of Trudinger–Moser type \eqref{TMinequality} when $\Omega=B$. Our next result concerns the attainability of supercritical inequality of Trudinger–Moser type \eqref{Sad2} for a general $f$ satisfying  \eqref{limsup} with $\sigma>2$.
\begin{theorem}\label{thm2}
Assume \eqref{limsup} with $\sigma>2$. 
Then the supremum $ \mathrm{MT}_{N}(f)$ is  attained for some $u\in \Phi^{k}_{0,\mathrm{rad}}(B)$.
\end{theorem}

As far as we know, the existence of maximizers for either \textit{Sobolev} or \textit{Trudinger-Moser} type supercritical inequalities was obtained only for the particular function $f(r)=r^{\alpha}$ with $ \alpha>0$, as shown in \cite{MR3514752, OLUBNA,Ngu,Ngu2}. Theorem~\ref{thm2} represents the first result to consider the existence of maximizers for a general class of functions $f$. 

As a consequence of Theorems \ref{thm1} and \ref{thm2}, by   choosing $f_0(r)=\gamma \, r^{a}/(1-r)^{b},\;  a >0$ and $b, \gamma \in \mathbb{R},$  we obtain the following:

\begin{corollary}\label{cor1}
Let $ a >0$ and $b, \gamma \in \mathbb{R}$. Then it holds
\begin{equation}\label{CorSad2}
\mathrm{MT}_{N}(f_0)=\sup_{u\in \Phi^{k}_{0,\mathrm{rad}}(B),\;\; \|u\|_{\Phi^{k}_{0}}=1}\int_{B}
\mathrm{e}^{
\mu_N|u|^{\frac{N+2}{N}+\frac{\gamma |x|^{a}}{\left(1-|x|\right)^{b}}
} }
\,\ud x<+\infty.
\end{equation}
In addition, if $\gamma>0$ then $\mathrm{MT}_{N}(f_0)$ is attained.
\end{corollary}

Note that if $\gamma \leq 0,$  \eqref{CorSad2} follows from the  Moser 
 type inequality  \eqref{TMinequality} and the case $\gamma>0$ is a consequence of Theorem~\ref{thm1}. The attainability of $\mathrm{MT}_{N}(f_0)$, with $\gamma, a >0$ and $b \in \mathbb{R}$
follows from Theorem~\ref{thm2}.
\section{Preliminaries}

Consider the weighted Sobolev space $X_1=X_1^{1,k+1}$ of all functions $v\in AC_{\mathrm{loc}}(0,1)$ satisfying
$$
\lim_{r\rightarrow 1}v(r)=0,\;\;\; \int_{0}^{1}r^{N-k}|v^{\prime}|^{k+1}\mathrm{d}r<\infty\;\;\;\mbox{and}\;\;  \int_{0}^{1}r^{N-1}|v|^{k+1}\mathrm{d}r<\infty,
$$
where $AC_{\mathrm{loc}}(0, 1)$ is the set of all locally absolutely continuous functions on $(0,1)$.  
The space $X_1$ is complete with the norm
\begin{equation}\label{fullnorm}
\|v\|_{X_1}=\left(c_N\int_{0}^{1}r^{N-k}|v^{\prime}|^{k+1}\mathrm{d}r\right)^{\frac{1}{k+1}},
\end{equation} 
where $c_{N}=(\omega_{N-1}{\binom{N}{k}})/N$ is a normalising constant, and  $\omega_{N-1}$ represents the area of the unit sphere in $\mathbb{R}^N$.

For reference, let us recall the following Sobolev embeddings. 
If $2k<N$, then the following continuous embedding holds
\begin{equation}\label{Ebeddings}
    X_1\hookrightarrow L^q_{N-1}, \quad \mbox{if}\quad q \in \left.\left[1, k^{*}\right.\right]\;\;\;\; \mbox{(compact if}\;\;q<k^{*}),
\end{equation}
where $k^*=N(k+1)/(N-2k)$ and  $L^q_{\theta}=L^q_{\theta}(0,1),\, q\ge 1,  \theta>-1$ is the weighted Lebesgue space of measurable functions $v$ on $(0,1)$ such that
\begin{equation}\nonumber
 \|v\|_{L^q_{\theta}}=
\left(\int_0^1 r^{\theta}|v|^q\,\mathrm{d}r
\right)^{\frac{1}{q}}<+\infty.
\end{equation}
For $2k=N$, we have the following compact embedding 
\begin{equation}\label{EbeddingsTM}
    X_1\hookrightarrow L^q_{\theta}, \;\; \mbox{if}\;\; q \in [1, \infty)
\end{equation}
for any $\theta\in (-1, \infty)$.
Hence, in contrast with \eqref{Ebeddings}, under the condition $N=2k$, the threshold growth is not achieved in any Lebesgue type space $L^q_{N-1}$. 
Indeed, from \cite[Theorem~1.1]{JJ2012} we have that $\exp\big(\mu \vert v\vert^{{(N+2)}/{N}}\big)\in L^{1}_{N-1}$ for any $\mu>0$ and $v\in X_1$ and  following Trudinger-Moser type inequality holds:
\begin{equation}\label{kPAMS-TM}
\mathcal{MT}_{N}=\sup_{\|v\|_{X_1}= 1}\int_{0}^{1}r^{N-1}\mathrm{e}^{\mu_N \vert v\vert^{\frac{N+2}{N}}}\mathrm{d}r\le c.
\end{equation}
In addition, the constant $\mu_N$ in \eqref{kPAMS-TM}  is sharp in the sense that if we replace $\mu_{N}$ by any $\beta>\mu_N$ the supremum becomes infinite.
For a deeper discussion and applications of the weighted Sobolev space $X_1$, we refer the reader to \cite{JDE2019} and references therein.

We can see that \eqref{admissible-norm} can be written as (cf. \cite{Wang2})
\begin{equation}\label{radial-norm}
\|u\|_{\Phi^{k}_{0}}=\left(c_{N}\int_{0}^{1}r^{N-k}\vert u^{\prime}\vert^{k+1}\ud r\right)^{\frac{1}{k+1}}, \quad \forall u\in\Phi^{k}_{0,\mathrm{rad}}(B).
\end{equation}
Note that taking $v(r)=u(x), \; r=\vert x\vert $, we  have  $v\in X_1,$
\begin{equation}\label{NormX}
\|u\|_{\Phi^{k}_0}=\|v\|_{X_1}
\end{equation}
and
\begin{equation}\label{FX2}
\int_{B}\mathrm{e}^{\mu_N|u|^{\frac{N+2}{N}+f(|x|)}}\ud x=\omega_{N-1}\int_{0}^{1} r^{N-1}\mathrm{e}^{\mu_N|v|^{\frac{N+2}{N}+f(r)}}\ud r
\end{equation}
where $f:[0,1)\rightarrow[0,\infty)$ is any continuous function. 

The important point to note here is that \eqref{NormX} and \eqref{FX2} yield the connection between the spaces  $ \Phi^{k}_{0,\mathrm{rad}}(B)$ and $X_1$. It allows us to analyze the supremum \eqref{Sad2} by means of 
 the  following auxiliary supremum on $X_1$ 
\begin{equation}\label{supremumX1} 
\mathcal{MT}_{N}(f):=\sup_{\|v\|_{X_1}=1}\int_{0}^{1} r^{N-1}\mathrm{e}^{\mu_{N}|v|^{\frac{N+2}{N}+f(r)}}\ud r.
\end{equation}

\section{Sharp supercritical inequalities: Proof of Theorem~\ref{thm1}}

In this section we prove that the supremum  $\mathrm{MT}_{N}(f)$ in \eqref{Sad2} is finite, which is reduced to show that  $\mathcal{MT}_{N}(f)$ is finite. Indeed, \eqref{NormX} and  \eqref{FX2} imply $\mathrm{MT}_{N}(f)\le \omega_{N-1}\mathcal{MT}_{N}(f)$.  

In order to analyze \eqref{supremumX1}, we start with the following radial type lemma.
\begin{lemma} Assume that $2k=N$. Then, for any $v \in X_1$
\begin{equation}\label{r-estimate}
|v(r)| \le \|v\|_{X_1}\big(-\frac{N}{\mu_{N}}\ln r\big)^{\frac{N}{N+2}},\;\; \mbox{for any}\;\; 0<r<1.
\end{equation}
\end{lemma}
\begin{proof} By using H\"{o}lder inequality we have  
\begin{equation*}
\begin{aligned}
|v(r)| \le \int_{r}^{1}|v^{\prime}|\ud s& =\int_{r}^{1}\big(c^{\frac{1}{k+1}}_{N}s^{\frac{N-k}{k+1}}|v^{\prime}|\big)\big( c^{-\frac{1}{k+1}}_{N}s^{-\frac{N-k}{k+1}}\big)\ud s\\
& \le \Big(\frac{N}{\mu_{N}}\Big)^{\frac{N}{N+2}}\|v\|_{X_1}\left(-\ln r\right)^{\frac{N}{N+2}},
\end{aligned}
\end{equation*}
where we have used $\mu_{N}=Nc_{N}^{2/N}$. 
\end{proof}
Following the line of \cite{MR3514752,Ngu2}, we are able to show the following:
\begin{proposition}\label{prop1}  Assume \eqref{limsup} with $\sigma>1$.
Then 
\begin{flushleft}
$(i)$ For any $\beta>0$ and $v\in X_1$, we have
$$
\int_{0}^{1} r^{N-1}\mathrm{e}^{\beta|v|^{\frac{N+2}{N}+f(r)}}\ud r<\infty.
$$
$(ii)$ If $\beta\le \mu_N$ then the following supremum is finite
\begin{equation}\label{supii}
\mathcal{MT}_{N}(f;\beta)=\sup_{\|v\|_{X_1}=1}\int_{0}^{1} r^{N-1}\mathrm{e}^{\beta|v|^{\frac{N+2}{N}+f(r)}}\ud r.
\end{equation}
$(iii)$ If $\beta > \mu_N$ then $\mathcal{MT}_{N}(f;\beta)=\infty.$
\end{flushleft}
\end{proposition}
\begin{proof} 
By simplicity, for  $v\in X_1$ we denote
$$
h(r):=h(r;v)=r^{N-1}e^{\beta \vert v(r)\vert^{\frac{N+2}{N}+f(r)}}.
$$
We first prove $(i)$, for that we split
\begin{equation*}
    \int_{0}^{1} h(r)\, \ud r = I_1 + I_2+I_3
\end{equation*}
where 
\begin{equation*}
     I_1 = \,  \int_{\rho}^{1} h(r) \, \ud r, \quad 
     I_2 = \, \int_{0}^{\hat{\rho}} h(r)\, \ud r, \quad 
     I_3 = \,  \int_{\hat{\rho}}^{\rho} h(r)\, \ud r 
\end{equation*}
and $0<\hat{\rho}<\rho<1$ will be chosen later. 
Note that since $h$ is continuous in $[\hat{\rho},\rho],$ we have   $h\in L^{1}(\hat{\rho}, \rho)$ independent of the choices of $\hat{\rho}$ and $\rho$. We claim that there are $0<\hat{\rho}=\hat{\rho}_v<\rho=\rho_v <1$ such that $|v(r)|^{f(r)}\le {3}/{2}$ for either  $r\in [\rho,1)$ or $r\in (0, \hat{\rho}]$.  Indeed,  $\lim_{r\to 1}v(r)=0$ ensures the existence of $\rho>0$ such that $|v|\le 1$ on $[\rho, 1)$. In addition, due to \eqref{limsup} and \eqref{r-estimate}, since $f(0)=0$ we obtain
\begin{equation}\nonumber
    \begin{aligned}
   |v|^{f(r)}&\le\big(\|v\|_{X_1}(-\frac{N}{\mu_{N}}\ln r)^{\frac{N}{N+2}}\big)^{f(r)}\\
    &=\mathrm{e}^{f(r)\ln\big(\|v\|_{X_1}(-\frac{N}{\mu_{N}}\ln r)^{\frac{N}{N+2}}\big)}\\
     &\to 1,\;\; \mbox{as}\;\; r\to 0^{+},
\end{aligned}
\end{equation}
which establish the existence of  $\hat{\rho}>0$. Hence, 
\begin{equation*}
  \mathrm{e}^{\beta|v|^{\frac{N+2}{N}+f(r)}} =\mathrm{e}^{\beta|v|^{\frac{N+2}{N}}|v|^{f(r)}} \le \mathrm{e}^{\frac{3\beta}{2}|v|^{\frac{N+2}{N}}}, 
\end{equation*}
for either  $r\in [\rho,1)$ or $r\in (0, \hat{\rho}]$. Since from \cite{JJ2012} we have $\exp\big(\mu \vert v\vert^{{(N+2)}/{N}}\big)\in L^{1}_{N-1}(0,1)$, for all $\mu>0$, the above estimate yields  $h\in L^{1}(0,\hat{\rho})$ and $h\in L^{1}(\rho, 1)$. 

\bigskip 

Now we prove $(ii)$, that is, $\mathcal{MT}_{N}(f;\beta)$ is finite for any $\beta\le \mu_N$.
Using estimate \eqref{r-estimate}, there exist $\rho, D >0$ (independent of $v$) such that 
\begin{equation}\label{a0log}
\begin{aligned}
|v(r)|\le (-\frac{N}{\mu_N}\ln r)^{\frac{N}{N+2}} \leq D (1-r)^{\frac{N}{N+2}} \leq 1, \quad \forall  r\in (\rho,1].
\end{aligned}
\end{equation}
Thus
\begin{equation*}
    |v(r)|^{f(r)} \leq 1, \quad \forall r \in [\rho,1],
\end{equation*}
which implies 
\begin{equation*}
        I_1 = \,  \int_{\rho}^{1} r^{N-1}\mathrm{e}^{\beta|v|^{\frac{N+2}{N}+f(r)}}\, \ud r \leq
        \int_{0}^{1} r^{N-1}\mathrm{e}^{\mu_N|v|^{\frac{N+2}{N}}}\, \ud r \leq \mathcal{MT}_{N}.
\end{equation*}

Now we will analyze the integral near the origin $r=0$. From $f(0)=0$ and  \eqref{limsup} we have $\frac{N}{N+2}f(r)\ln(-\frac{N}{\mu_{N}}\ln r)\to 0$ as $r\to 0$. Thus, \eqref{r-estimate} and the series expansion $\mathrm{e}^{x}=1+x+o(x)$, as $x\rightarrow 0$ yield 
\begin{equation}\nonumber
    \begin{aligned}
   |v|^{f(r)}&\leq \left[\left(-\frac{N}{\mu_{N}}\ln r\right)^{\frac{N}{N+2}}\right]^{f(r)}\\
    &=\mathrm{e}^{\frac{N}{N+2}f(r)\ln(-\frac{N}{\mu_{N}}\ln r)}\\
    &=1+\frac{N}{N+2}f(r)\ln(-\frac{N}{\mu_{N}}\ln r)+ o\big(f(r)\ln(-\frac{N}{\mu_{N}}\ln r)\big)
    \quad  \text{as} \quad  r \rightarrow 0 .
\end{aligned}
\end{equation}
For $\hat{\rho}>0$ small enough, we have $-N\ln r\ge 1 $, for $r\in (0,\hat{\rho})$. Thus,  from  \eqref{r-estimate} we have
\begin{multline}\label{KR}
    {\mu_{N}|v|^{\frac{N+2}{N}}\big(\big(|v|^{f(r)}-1\big)}\leq \\
    (-N \ln r)\Big[\frac{N}{N+2}f(r)\ln(-\frac{N}{\mu_{N}}\ln r)+ o\big(f(r)\ln(-\frac{N}{\mu_{N}}\ln r)\big)\Big]
    \quad  \text{as} \quad  r \rightarrow 0.
\end{multline}
Hence, from \eqref{KR} and  \eqref{limsup}, there are $C,\hat{\rho}>0$ (independent of $v$) such that 
\begin{equation*}
    \mathrm{e}^{\mu_{N}|v|^{\frac{N+2}{N}}\big(|v|^{f(r)}-1\big)}\le C, \quad \forall r  \in (0,\hat{\rho}].
\end{equation*}
We obtain
\begin{equation}\nonumber
\begin{aligned}
 I_2\le \int_{0}^{\hat{\rho}} r^{N-1}\mathrm{e}^{\mu_{N}|v|^{\frac{N+2}{N}+f(r)}}\ud r & =\int_{0}^{\hat{\rho}} r^{N-1}\mathrm{e}^{\mu_{N}|v|^{\frac{N+2}{N}}}\Big[\mathrm{e}^{\mu_{N}|v|^{\frac{N+2}{N}}\big( |v|^{f(r)}-1\big)}\Big]\ud r\\
 & \le C\int_{0}^{\hat{\rho}} r^{N-1}\mathrm{e}^{\mu_{N}|v|^{\frac{N+2}{N}}}\ud r\\
 &\le C \mathcal{MT}_{N}.
\end{aligned}
\end{equation}
Finally,  \eqref{r-estimate} and the continuity of the function $f$ on $[\hat{\rho}, \rho] $ allow us to estimate $I_3$.

In order to complete the proof of Proposition~\ref{prop1}, we only need to show that the threshold $\mu_{N}$ in \eqref{supii} is sharp. To see this, we will use the Moser's  functions
\begin{equation}\label{Moser-sequence}
 w_j(r)=\frac{1}{\mu^{\frac{N}{N+2}}_N}\left\{\begin{aligned}
 & j^{\frac{N}{N+2}}, &\;\;\mbox{if}\;\;& 0< r\le \mathrm{e}^{-\frac{j}{N}}\\
  & {j}^{-\frac{2}{N+2}}N\ln\left(\frac{1}{r}\right), &\;\;\mbox{if}\;\;& \mathrm{e}^{-\frac{j}{N}}\le r\le 1.
 \end{aligned}\right.
 \end{equation}
It is easy to check that
 $\|w_{j}\|_{X_1}=1$, with $N=2k$.  Also, since $f\ge0$ and $\vert w_j\vert\ge 1$  on $(0, \mathrm{e}^{-j/N})$, for $j$ sufficiently large, if $\beta>\mu_{N}$ we have
 \begin{equation}\nonumber
 \begin{aligned}
 & \int_{0}^{1} r^{N-1}\mathrm{e}^{\beta\vert w_j\vert^{\frac{N+2}{N}+f(r)}}\ud r \ge  \int_{0}^{\mathrm{e}^{-j/N}} r^{N-1}\mathrm{e}^{\beta\vert w_j\vert^{\frac{N+2}{N}}}\ud r
= \frac{\mathrm{e}^{j(\beta\mu^{-1}_N-1)}}{N}\rightarrow+\infty \quad \text{as} \quad j\rightarrow\infty.
 \end{aligned}
 \end{equation}
\end{proof}
\begin{proof}[Proof Theorem~\ref{thm1}:]
Follows from \eqref{NormX}, \eqref{FX2}  and Proposition~\ref{prop1}.
\end{proof}
 
\section{Existence of extremals for auxiliary problem}
\label{sec3}
We shall analyze the attainability of the supremum $\mathcal{MT}_{N}(f)$. Following \cite{CC,FOR,MZ2020} we say that $(v_j)\subset X_{1}$ is a normalized concentrating sequence at origin, NCS for short,   if 
\begin{equation}\label{def-concentra}
\|v_j\|_{X_1}=1,\;\; v_j\rightharpoonup 0\;\;\mbox{weakly in}\;\; X_1 \;\;\mbox{and}\;\; \lim_{j\rightarrow 0}\int_{r_0}^{1}r^{N-k}|v^{\prime}_j|^{k+1}dr=0,\;\forall r_0>0.
\end{equation}
Let us denote by $\Upsilon_0$ the set of normalized concentrating sequences at the origin and define the concentration level by
\begin{equation}
J=\sup\Big\{\limsup_{j\rightarrow\infty}\int_{0}^{1}r^{N-1}\mathrm{e}^{\mu_{N}|v_j|^{\frac{N+2}{N}}}dr\;:\; (v_j)\in  \Upsilon_0 \Big\}.
\end{equation}
Next, we will obtain an upper bound for the concentrating level  $J$, which relies on the following technical lemma proved in \cite[Lemma~2.3]{JJ2012} (see also \cite[Lemma~2]{CC}).
\begin{lemma}\label{sharp-estimate} Let $p\ge2$, $a>0$ and $\delta>0$ be real numbers. For each $w$ non-negative
$C^{1}$ function on $[0, \infty)$ satisfying $\int_{a}^{\infty}|w^{\prime}(t)|^{p}\mathrm{d}t \le \delta$ we have
$$
\int_{a}^{\infty}\mathrm{e}^{w^{q}(t)-t}\mathrm{d}t\le \mathrm{e}^{w^{q}(a)-a}\frac{1}{1-\delta^{1/(p-1)}}\exp\bigg(\Big(\frac{p-1}{p}\Big)^{p-1}\frac{c^{p}\gamma_{p}}{p}+\Psi(p)+\gamma\bigg)
$$
where $\gamma_{p}=\delta(1-\delta^{1/(p-1)})^{1-p}$, $c=qw^{q-1}(a)$, $q=p/(p-1)$, $\Psi(x)={\Gamma}^{\prime}(x)/\Gamma(x)$ with $\Gamma(x)=\int_{0}^{1}(-\ln t)^{x-1}dt$, $x>0$ is the digamma function and $\gamma=-\Psi(1)$ is the  Euler-Mascheroni constant.
\end{lemma}

Using Lemma~\ref{sharp-estimate} we can obtain the following estimate 
\begin{lemma}\label{S-lemma} We have
\begin{equation}\label{b-c} 
J \le \frac{1}{N}\Big(1+\mathrm{e}^{\Psi(\frac{N}{2}+1)+\gamma}\Big).
\end{equation}
\end{lemma}
\begin{proof}
Let $(v_j)\in  \Upsilon_0$. Since $\Psi(\frac{N}{2}+1)+\gamma >0$, if 
\begin{equation*}
    \limsup_{j\rightarrow\infty}\int_{0}^{1}r^{N-1}\mathrm{e}^{\mu_{N}|v_j|^{\frac{N+2}{N}}}\mathrm{d}r \leq \frac{2}{N},
\end{equation*}
then the inequality \eqref{b-c}  holds. Thus,  without loss of generality, 
we can assume that $(v_j)\in  \Upsilon_0$ is such that 
\begin{equation}\label{hc}
\limsup_{j\rightarrow\infty}\int_{0}^{1}r^{N-1}\mathrm{e}^{\mu_{N}|v_j|^{\frac{N+2}{N}}}\mathrm{d}r> \frac{2}{N}.
\end{equation}
\par
We claim that for each $j\in\mathbb{N}$ there is a largest number $r_j\in (0,\mathrm{e}^{-1/N}]$  such that 
\begin{equation}\label{r-solution}
\mu_N|v_{j}(r_j)|^{\frac{N+2}{N}}=-N\ln r_{j}-2\ln^{+}\left(-N\ln r_{j}\right),
\end{equation}
where $f^{+}=\max\left\{f,0\right\}$.  In fact, from \eqref{r-estimate}, we have
\begin{equation}\nonumber
\mu_{N}|v_{j}(r)|^{\frac{N+2}{N}}\le -N\ln r=-N\ln r-2\ln^{+}\left(-N\ln r\right),\;\;\; r\ge \mathrm{e}^{-1/N}.
\end{equation}
By contradiction, suppose that there is no number $r_j$ satisfying \eqref{r-solution}, then
\begin{equation}\nonumber
\mu_{N}|v_{j}(r)|^{\frac{N+2}{N}} < -N\ln r-2\ln^{+}\left(-N\ln r\right),\;\;\; r\in (0,1).
\end{equation}
Note that
\begin{eqnarray}
\int_{0}^{1}r^{N-1}\mathrm{e}^{\mu_{N}|v_j|^{\frac{N+2}{N}}}\mathrm{d}r&\le & \int_{0}^{\mathrm{e}^{-1/N}}r^{N-1}\mathrm{e}^{-N\ln r-2\ln^{+}\left(-N\ln r\right)}\mathrm{d}r\nonumber
+\int_{\mathrm{e}^{-1/N}}^{1}r^{N-1}\mathrm{e}^{-N\ln r}\mathrm{d}r\nonumber\\
&\le & \frac{1}{N}\left(\int_{1}^{\infty}\frac{1}{t^2}\mathrm{d}t+1\right)\nonumber\\
&=&\frac{2}{N},\nonumber
\end{eqnarray}
which contradicts \eqref{hc}. Therefore, there exists a sequence  $(r_j)$ satisfying \eqref{r-solution}. We claim that 
\begin{equation}\label{limitrn}
\lim_{j\rightarrow\infty}r_{j}=0.
\end{equation}
This follows by the concentration assumption on $(v_j)$. Indeed, using \eqref{def-concentra}, given $\epsilon>0$ there exists $j_0=j_0(\epsilon)$ such that 
$$
\left(c_{N}\int_{\epsilon}^{1}r^{N-k}|v^{\prime}_j|^{k+1}\mathrm{d}r\right)^{\frac{2}{N}}<\frac {\mathrm{e}-2} {\mathrm{e}},\;\;\; j\ge j_0.
$$
Thus, by the H\"{o}lder inequality we have for any $r\in (\epsilon, 1)$ and $j\ge j_0$
\begin{eqnarray*}
\mu_{N}|v_{j}(r)|^{\frac{N+2}{N}}&\le & \left(c_{N}\int_{\epsilon}^{1}r^{N-k}|v^{\prime}_{j}|^{k+1}\mathrm{d}r\right)^{\frac{2}{N}}(-N\ln r)\\
&<& \frac {\mathrm{e-2}}{\mathrm{e}}\,(-N \ln r)\\
&\le & -N\ln r-2\ln^{+}\Big(-N\ln r\Big),
\end{eqnarray*}
which implies that $r_j<\epsilon$ for $j\ge j_0,$ and consequently \eqref{limitrn} holds. \\
\par \smallskip
Performing the change of variables
$$
t=-N\ln r, \;\;\; w_{j}(t)=c_{N}^{\frac{1}{k+1}}N^{\frac{N}{N+2}}v_{j}(r)\;\;\;\mbox{and}\;\;\; a_{j}=-N\ln r_j,
$$
we can write
$$
c_{N}\int_{0}^{r_j}r^{N-k}|v^{\prime}_{j}|^{k+1}\mathrm{d}r=\int_{a_{j}}^{\infty}|w^{\prime}_{j}(t)|^{k+1}\mathrm{d}t.
$$
We are in a position to apply Lemma~\ref{sharp-estimate} with
$$
w=w_j,\;\; a=a_j\;\;\; \mbox{and}\;\;\; \delta=\delta_{j}= \int_{a_{j}}^{\infty}|w^{\prime}_{j}(t)|^{k+1}\mathrm{d}t
$$
to obtain
\begin{equation}\label{concentration-boundeness}
\int_{0}^{r_j}r^{N-1}\mathrm{e}^{\mu_{N}|v_{j}|^{\frac{N+2}{N}}}\ud r=\frac{1}{N}\int_{a_j}^{\infty}\mathrm{e}^{|w_{j}(t)|^{\frac{N+2}{N}}-t}\mathrm{d}t\le \frac{1}{N}\frac{\mathrm{e}^{K_{j}+\Psi(k+1)+\gamma}}{1-\delta^{\frac{2}{N}}_{j}} 
\end{equation}
where
$$
K_{j}=w^{\frac{N+2}{N}}_j(a_j)\bigg[1+ \frac{\delta_{j}}{\frac{N}{2}(1-\delta^{\frac{2}{N}}_{j})^{\frac{N}{2}}}\bigg]- a_j\ .
$$

We claim that
\begin{equation}\label{blabli}
    \delta_{j}\rightarrow 0\;\;\; \mbox{and}\;\;\; K_{j}\rightarrow 0,\;\;\;\mbox{as}\;\;\; j\rightarrow\infty.
\end{equation}
In fact, by the same way as in \eqref{r-estimate}, we have 
\begin{eqnarray*}
|w_{j}(a_j)|^{\frac{N+2}{N}}=\mu_{N}|v_{j}(r_j)|^{\frac{N+2}{N}}
&\le & \left(c_{N}\int_{r_j}^{1}s^{N-k}|v^{\prime}_{j}(s)|^{k+1}\mathrm{d}s\right)^{\frac{2}{N}}(-N\ln r_j)\\
&\le & (1-\delta_{j})^{\frac{2}{N}}a_j.
\end{eqnarray*}
We recall that $1- t^d \le d(1-t)$ for $t\ge 0 $ and  $d \geq 1$. 
Hence, by using the above inequality, \eqref{r-solution} and \eqref{limitrn} we obtain
$$\delta_{j}\le 1-\left(1-\frac{2\ln^{+}a_j}{a_j}\right)^{\frac{N}{2}}\le N\frac{\ln^{+}a_j}{a_j}\rightarrow 0,$$ as $j\rightarrow \infty.$
 In addition, note that 
$$K_j=a_j(1-\delta^{\frac{2}{N}}_{j})^{-\frac{N}{2}}G(a_j),$$
where
$$
G(a_j)=\frac{2\delta_{j}}{N}-\frac{2\delta_{j}}{N}\frac{2\ln^{+}a_j}{a_j}-\frac{2\ln^{+}a_j}{a_j}(1-\delta^{\frac{2}{N}}_{j})^{\frac{N}{2}}.
$$
Since $\delta_{j}\le N\frac{\ln^{+}a_j}{a_j}$, by the inequality $1- t^d \le d(1-t)$ for $t\ge 0 $ and  $d \geq 1$ again we conclude that
\begin{eqnarray*}
G(a_j)&\le &  \frac{2\ln^{+}a_j}{a_j}\left[ 1-\left(1-\delta_{j}^{\frac{2}{N}}\right)^{\frac{N}{2}}-\frac{2\delta_{j}}{N}\right]\\
&\le& \frac{2\ln^{+}a_j}{a_j}\left[\frac{N}{2}\delta_{j}^{\frac{2}{N}}-\frac{2\delta_{j}}{N}\right]\\
&\le& \frac{2\ln^{+}a_j}{a_j}\left(\frac{N}{2}\delta_{j}^{\frac{2}{N}}\right)\\
&\le&\left (\frac{N}{2}\right)^{\frac{N+2}{N}}\left(\frac{2\ln^{+}a_j}{a_j}\right)^{\frac{N+2}{N}},
\end{eqnarray*}
which implies 
\begin{equation}\nonumber
K_j=\frac{a_j}{(1-\delta^{\frac{2}{N}}_{j})^{\frac{N}{2}}}G(a_j)\le \frac{\left (\frac{N}{2}\right)^{\frac{N+2}{N}}}{(1-\delta^{\frac{2}{N}}_{j})^{\frac{N}{2}}}a_j\left(\frac{2\ln^{+}a_j}{a_j}\right)^{\frac{N+2}{N}}
\end{equation}
and thus $K_j\rightarrow 0$, as $j\rightarrow\infty,$ and consequently \eqref{blabli} holds. 
\par \medskip
Hence, from \eqref{concentration-boundeness}
\begin{equation}\label{concentration-boundeness2}
\limsup_{j\rightarrow\infty}\int_{0}^{r_j}r^{N-1}\mathrm{e}^{\mu_{N}|v_{j}|^{\frac{N+2}{N}}}\mathrm{d}r\le \frac{1}{N}\mathrm{e}^{\Psi(k+1)+\gamma}.
\end{equation}
We recall that for any $r_0>0$
\begin{eqnarray*}
\mu_{N}|v_{j}(r)|^{\frac{N+2}{N}}
&\le & \left(c_{N}\int_{r_0}^{1}s^{N-k}|v^{\prime}_{j}(s)|^{k+1}\mathrm{d}s\right)^{\frac{2}{N}}(-N\ln{r_0}),\;\; \forall\; r\in [r_0,1].
\end{eqnarray*}
It follows from \eqref{def-concentra} that   $v_{j}\rightarrow 0$ uniformly on every compact interval $[r_0,1]\subset (0,1]$. In particular, given $\epsilon>0$ and $r_0\in (r_j, \mathrm{e}^{-1/N})$, there exists $j_0\in\mathbb{N}$ satisfying
$$
\mu_{N}|v_{j}(r)|^{\frac{N+2}{N}}
\le\epsilon,\;\;  \forall\; r\in [r_0,1] \;\;\;\mbox{and} \;\;\; j\ge j_0.
$$
Recalling that $r_j\in (0, \mathrm{e}^{-1/N}]$ is the largest number for which \eqref{r-solution} holds we get
\begin{eqnarray*}
\int_{r_j}^{1}r^{N-1}\mathrm{e}^{\mu_{N}|v_{j}|^{\frac{N+2}{N}}}\mathrm{d}r&=&\int_{r_j}^{r_0}r^{N-1}\mathrm{e}^{\mu_{N}|v_{j}|^{\frac{N+2}{N}}}\mathrm{d}r+\int_{r_0}^{1}r^{N-1}\mathrm{e}^{\mu_{N}|v_{j}|^{\frac{N+2}{N}}}\mathrm{d}r\\
&\le &\int_{r_j}^{r_0}r^{N-1}\mathrm{e}^{-N\ln r-2\ln^{+}\left(-N\ln r\right)}\mathrm{d}r+\mathrm{e}^{\epsilon}\int_{r_0}^{1}r^{N-1}\mathrm{d}r\\
&=& \frac{1}{N}\left[-\frac{1}{(-N\ln r_j)}+\frac{1}{(-N\ln r_0)}+\mathrm{e}^{\epsilon}-\mathrm{e}^{\epsilon}r^{N}_0\right]\\
&\le & \frac{1}{N}
\end{eqnarray*}
for $\epsilon$ and $r_0$ small enough. Thus,  combining with \eqref{concentration-boundeness2}, we have  that \eqref{b-c} holds.
\end{proof}
Next, we will show that 
\begin{equation}\nonumber
\mathcal{MT}_{N}(f)>\frac{1}{N}\Big(1+\mathrm{e}^{\gamma+\Psi\big(\frac{N}{2}+1\big)}\Big)
\end{equation}
for any continuous function $f:[0,1)\rightarrow[0,\infty)$ satisfying $f(0)=0$, $f>0$ on $(0,1)$ and the condition \eqref{limsup}. 
To obtain the above estimate we use the following family of functions: For $\epsilon>0$ small enough, let
\begin{equation}\label{sv_e}
v_{\epsilon}(r)=\left\{\begin{aligned}
& c_{\epsilon}-\frac{\frac{N}{2\mu_{N}}\ln\left(1+ a_{N}\left(\frac{r}{\epsilon}\right)^{2}\right)+b_{\epsilon}}{c^{\frac{2}{N}}_{\epsilon}}, & & r\le \epsilon L_{\epsilon}\\
&\frac{(-N\ln r)}{\mu_{N}c^{\frac{2}{N}}_{\epsilon}},  & & \epsilon L_{\epsilon}\le r\le 1, 
\end{aligned}\right.
\end{equation}
where $ a_{N}=\left(\frac{\omega_{N-1}}{N}\right)^{\frac{2}{N}}$, $L_{\epsilon}=-\ln\epsilon$ and $b_{\epsilon}, c_{\epsilon} $ are suitable constants to be chosen later such that $\|v_{\epsilon}\|_{X_1}=1$ and the function $v_{\epsilon}$ to be continuous, that is,
\begin{enumerate}
\item [(i)] $c_{\epsilon} \rightarrow\infty$, as $\epsilon\rightarrow 0$,
\item [(ii)]$ 2\mu_{N} c_{\epsilon}^{\frac{N+2}{N}}=
N\ln\left(1+ a_{N}L^{2}_{\epsilon}\right)+2\mu_{N} b_{\epsilon}-2N\ln({\epsilon L_{\epsilon}})$.
\end{enumerate}
Before we perform our analysis of the family $(v_{\epsilon})$, we need to introduce the following technical result.
\begin{lemma}\label{lema-tecnico} For any $z>0$ and $p\ge2$ we have 
\begin{equation}\label{LT-eq1}
\int_{0}^{z}\frac{s^{p-1}}{(1+s)^{p}}\mathrm{d}s=\ln(1+z)-[\gamma+\Psi(p)]+\int_{\frac{z}{1+z}}^{1}\frac{1-s^{p-1}}{1-s}\mathrm{d}s
\end{equation}
and
\begin{equation}\label{LT-eq2}
\int_{0}^{z}\frac{s^{p-2}}{(1+s)^{p}}\mathrm{d}s=p-1-\int_{z}^{\infty}\frac{1-s^{p-2}}{1-s}\mathrm{d}s.
\end{equation}
\end{lemma}
\begin{proof}
We recall (cf. \cite{special}) the following identities
\begin{equation}\label{GPsi}
\Gamma(1)=1, \;\;\Gamma(x+1)=x\Gamma(x)\;\; \mbox{and}\;\;  
\int_{0}^{\infty}\frac{s^{x-1}}{(1+s)^{x+y}}ds=\frac{\Gamma(x)\Gamma(y)}{\Gamma(x+y)},\;\; x, y>0
\end{equation}
and the Dirichlet's integral representation of $\Psi(x)$
\begin{equation}\label{DPsi}
\Psi(x)=\int_{0}^{\infty}\frac{1}{z}\left(1-\frac{1}{(1+z)^x}\right)\mathrm{d}z,\quad x>0.
\end{equation}
Since $\Psi(1)=-\gamma$, by using  \eqref{DPsi} and the change of variables $s=1/(1+z)$ it follows that
\begin{equation}\nonumber
\begin{aligned}
\Psi(p)+\gamma&=\int_{0}^{\infty}\frac{1}{z}\left[\left(1-\frac{1}{(1+z)^p}\right)-\left(1-\frac{1}{1+z}\right)\right]\mathrm{d}z\\
&=\int_{0}^{1}\frac{1-s^{p-1}}{1-s}\mathrm{d}s.
\end{aligned}
\end{equation}
Thus, setting $1/\tau =1+s$  we can write
\begin{equation}\nonumber
\begin{aligned}
\int_{0}^{z}\frac{s^{p-1}}{(1+s)^{p}}\mathrm{d}s
&= \int_{\frac{1}{1+z}}^{1}\left[\frac{1}{\tau}+\frac{1}{\tau}((1-\tau)^{p-1}-1)\right]\mathrm{d}\tau\\
&=\ln(1+z)-[\Psi(p)+\gamma]+\int_{\frac{z}{1+z}}^{1}\frac{1-s^{p-1}}{1-s}\mathrm{d}s
\end{aligned}
\end{equation}
which proves \eqref{LT-eq1}. Analogously, with help of \eqref{GPsi} we obtain \eqref{LT-eq2}. 
\end{proof}
\begin{lemma}  For $\epsilon>0$ small enough, we have   $v_{\epsilon}\in X_1$ with
\begin{equation}\label{gradient-full}
\begin{aligned}
\|v_{\epsilon}\|^{k+1}_{X_1}&=\frac{N}{2\mu_{N}c^{\frac{N+2}{N}}_{\epsilon}}\left[\ln(1+a_{\epsilon})-\Psi\Big(\frac{N}{2}+1\Big)-\gamma+O\left(\frac{1}{1+a_{\epsilon}}\right)-2\ln( \epsilon L_{\epsilon})\right],
\end{aligned}
\end{equation}
as $\epsilon\to 0$,
where $a_{\epsilon}=a_{N}L^{2}_{\epsilon}$. In particular, we can choose  $b_{\epsilon}$ and  $c_{\epsilon}$ , with $c_{\epsilon}$ of the form
\begin{equation}\label{c}
\begin{aligned}
\mu_{N}c^{\frac{N+2}{N}}_{\epsilon}=\frac{N}{2}\left[\ln(1+a_{\epsilon})-\Psi\Big(\frac{N}{2}+1\Big)-\gamma+O\left(\frac{1}{1+a_{\epsilon}}\right)-2\ln( \epsilon L_{\epsilon})\right]
\end{aligned}
\end{equation}
such that  $\|v_{\epsilon}\|_{X_1}=1$, $(i)$ and $(ii)$ hold.
\end{lemma}
\begin{proof}
Since we are assuming $N=2k$, an straightforward calculation shows  
\begin{equation}\label{part1-gradient}
c_{N}\int_{0}^{\epsilon L_{\epsilon}}r^{N-k}|v^{\prime}_{\epsilon}|^{k+1}\mathrm{d}r= \frac{N}{2\mu_{N}c^{\frac{N+2}{N}}_{\epsilon}}\int_{0}^{a_N L^{2}_{\epsilon}}\frac{s^{\frac{N}{2}}}{(1+s)^{\frac{N}{2}+1}}\mathrm{d}s.
\end{equation}
The  L’Hospital's rule gives
\begin{equation}\nonumber
\lim_{z\rightarrow\infty}(1+z)\int_{\frac{z}{1+z}}^{1}\frac{1-s^{p-1}}{1-s}\mathrm{d}s=p-1.
\end{equation} 
Hence, from \eqref{LT-eq1}  and \eqref{part1-gradient} it follows that
\begin{equation}\label{part1-gradientbis}
c_{N}\int_{0}^{\epsilon L_{\epsilon}}r^{N-k}|v^{\prime}_{\epsilon}|^{k+1}\mathrm{d}r=\frac{N}{2\mu_{N}c^{\frac{N+2}{N}}_{\epsilon}}\left[\ln(1+a_{\epsilon})-\Psi\Big(\frac{N}{2}+1\Big)-\gamma+O\left(\frac{1}{1+a_{\epsilon}}\right)\right].
\end{equation}
 On the other hand
\begin{eqnarray}
c_{N}\int_{\epsilon L_{\epsilon}}^{1}r^{N-k}|v^{\prime}_{\epsilon}|^{k+1}\mathrm{d}r &=& c_{N}\left(\frac{N}{\mu_{N}c^{\frac{2}{N}}_{\epsilon}}\right)^{k+1}\int_{\epsilon L_{\epsilon}}^{1}r^{-1}\mathrm{d}r\nonumber \\
&=&- \frac{N}{\mu_{N}c^{\frac{N+2}{N}}_{\epsilon}}\ln({\epsilon L_{\epsilon}}).\label{part2-gradient}
\end{eqnarray}
Now, the result follows from \eqref{part1-gradientbis} and \eqref{part2-gradient}.
\end{proof}

Now, we are in position to prove the following:
\begin{lemma}  We have
\begin{equation}\label{lower-bound}
\mathcal{MT}_{N}(f)>\frac{1}{N}\Big(1+\mathrm{e}^{\gamma+\Psi\big(\frac{N}{2}+1\big)}\Big).
\end{equation}
\end{lemma}
\begin{proof}
We will use the sequence $v_{\epsilon}$ given by \eqref{sv_e}. From  \eqref{c}, in order to satisfy $(ii)$  we must have  $b_{\epsilon}$ of the form
\begin{equation}\label{LambdaE}
\begin{aligned}
\mu_{N}b_{\epsilon}&= \mu_{N} c^{\frac{N+2}{N}}_{\epsilon}-\frac{N}{2}\ln\left(1+ a_{\epsilon}\right)+N\ln{\epsilon L_{\epsilon}}
&= -\frac{N}{2}\big[\Psi\Big(\frac{N}{2}+1\Big)+\gamma\big]+O\left(\frac{1}{1+a_{\epsilon}}\right).
\end{aligned}
\end{equation}
Thus, the elementary inequality  $(1+t)^{\frac{N+2}{N}}\ge 1+\frac{N+2}{N}t,\; t>-1$ gives 
\begin{eqnarray*}
 \mu_{N}|v_{\epsilon}|^{\frac{N+2}{N}}&=&\mu_{N}c^{\frac{N+2}{N}}_{\epsilon}\left| 1-\frac{\frac{N}{2\mu_{N}}\ln\left(1+ a_{N}\left(\frac{r}{\epsilon}\right)^{2}\right)+b_{\epsilon}}{c_{\epsilon}^{\frac{N+2}{N}}}\right|^{\frac{N+2}{N}}\\
 &\ge &\mu_{N} c^{\frac{N+2}{N}}_{\epsilon}- \left(\frac{N}{2}+1\right)\ln\left(1+ a_{N}\left(\frac{r}{\epsilon}\right)^{2}\right)-\left(1+\frac{2}{N}\right)\mu_{N}b_{\epsilon}\\
 &=& \Psi\Big(\frac{N}{2}+1\Big)+\gamma -N\ln{\epsilon L_{\epsilon}}+\frac{N}{2}\ln(1+a_{\epsilon})+O\left(\frac{1}{1+a_{\epsilon}}\right)\\
 &-&\left(\frac{N}{2}+1\right)\ln\left(1+ a_{N}\left(\frac{r}{\epsilon}\right)^{2}\right),
\end{eqnarray*}
for any $r\in (0,\epsilon L_{\epsilon})$.
Hence, using \eqref{GPsi} with $x=\frac{N}{2}$ and $y=1$
\begin{equation}\nonumber
\begin{aligned}
& \int_{0}^{\epsilon L_{\epsilon}}r^{N-1}\mathrm{e}^{ \mu_{N}|v_{\epsilon}|^{\frac{N+2}{N}}}\mathrm{d}r\\
& \ge \frac{\epsilon^{N}}{2a_{N}^{\frac{N}{2}}}\Big(\mathrm{e}^{\Psi\big(\frac{N}{2}+1\big)+\gamma -N\ln{\epsilon L_{\epsilon}}+\frac{N}{2}\ln(1+a_{\epsilon})+O\left(\frac{1}{1+a_{\epsilon}}\right)}\Big)\Big(\int_{0}^{a_{\epsilon}}\frac{s^{\frac{N}{2}-1}}{\left(1+s\right)^{\frac{N}{2}+1}}\mathrm{d}s\Big)\\
&=\frac{\epsilon^{N}}{2a_{N}^{\frac{N}{2}}}\Big(\mathrm{e}^{\Psi\big(\frac{N}{2}+1\big)+\gamma -N\ln{\epsilon L_{\epsilon}}+\frac{N}{2}\ln(1+a_{\epsilon})+O\left(\frac{1}{1+a_{\epsilon}}\right)}\Big)\Big(\frac{2}{N}+O\Big(\frac{1}{a_{\epsilon}}\Big)\Big)\\
\end{aligned}
\end{equation}
because (L'Hospital's rule)
$$
\lim_{x\rightarrow\infty}x\left(\int_{x}^{\infty}\frac{s^{\frac{N}{2}-1}}{\left(1+s\right)^{\frac{N}{2}+1}}\mathrm{d}s\right)=1.
$$
Therefore, using $\mathrm{e}^{t}\ge 1+t,\;\; t\in\mathbb{R}$
\begin{equation}\nonumber
\begin{aligned}
\int_{0}^{\epsilon L_{\epsilon}}r^{N-1}\mathrm{e}^{\mu_{N}|v_{\epsilon}|^{\frac{N+2}{N}}}\mathrm{d}r&\ge \frac{\epsilon^{N}}{Na_{N}^{\frac{N}{2}}}\Big(\mathrm{e}^{\Psi\big(\frac{N}{2}+1\big)+\gamma -N\ln{\epsilon L_{\epsilon}}+\frac{N}{2}\ln(1+a_{\epsilon})+O\left(\frac{1}{1+a_{\epsilon}}\right)}\Big)\Big(1+O\Big(\frac{1}{a_{\epsilon}}\Big)\Big)\\
&= \frac{1}{N}\Big(1+\frac{1}{a_{\epsilon}}\Big)^{\frac{N}{2}}\Big(\mathrm{e}^{\Psi\big(\frac{N}{2}+1\big)+\gamma+O\left(\frac{1}{1+a_{\epsilon}}\right)}\Big)\Big(1+O\Big(\frac{1}{a_{\epsilon}}\Big)\Big)\\
&\ge  \frac{1}{N}\Big(1+\frac{1}{a_{\epsilon}}\Big)^{\frac{N}{2}}\Big(\mathrm{e}^{\Psi\big(\frac{N}{2}+1\big)+\gamma}\Big)\Big(1+O\Big(\frac{1}{1+a_{\epsilon}}\Big)\Big)\Big(1+O\Big(\frac{1}{a_{\epsilon}}\Big)\Big)\\
&=\frac{1}{N}\Big(1+\frac{1}{a_{\epsilon}}\Big)^{\frac{N}{2}}\Big(\mathrm{e}^{\Psi\big(\frac{N}{2}+1\big)+\gamma}\Big)+O\Big(\frac{1}{a_{\epsilon}}\Big).
\end{aligned}
\end{equation}
It follows that 
\begin{equation}\label{Sb-part1}
\int_{0}^{\epsilon L_{\epsilon}}r^{N-1}\mathrm{e}^{\mu_{N}|v_{\epsilon}|^{\frac{N+2}{N}}}\mathrm{d}r\ge \frac{1}{N}\mathrm{e}^{\Psi\big(\frac{N}{2}+1\big)+\gamma}+O\Big(\frac{1}{a_{\epsilon}}\Big).
\end{equation}
On the other hand,  using $\mathrm{e}^{t}\ge 1+\frac{t^{k}}{k!},\;\; t\ge0$
\begin{equation}\label{Sb-part2}
\begin{aligned}
&\int_{\epsilon L_{\epsilon}}^{1}r^{N-1}\mathrm{e}^{\mu_{N}|v_{\epsilon}|^{\frac{N+2}{N}}}\mathrm{d}r\ge \int_{\epsilon L_{\epsilon}}^{1}r^{N-1}\Big(1+\frac{\mu^{k}_N}{k!}|v_{\epsilon}|^{k\frac{N+2}{N}}\Big)\mathrm{d}r\\
&=\frac{1}{N}+\frac{1}{c^{\frac{N+2}{N}}_{\epsilon}}\Big[\frac{1}{k!(\mu_{N}N)}\int_{0}^{-N\ln(\epsilon L_{\epsilon})}s^{\frac{N}{2}+1}\mathrm{e}^{-s}\mathrm{d}s+  O\Big(c^{\frac{N}{N+2}}_{\epsilon}(\epsilon L_{\epsilon})^{N}\Big)\Big].
\end{aligned}
\end{equation}
From \eqref{Sb-part1} and \eqref{Sb-part2}, we get 
\begin{equation}\label{Sb-total}
\begin{aligned}
&\int_{0}^{1}r^{N-1}\mathrm{e}^{\mu_{N}|v_{\epsilon}|^{\frac{N+2}{N}}}\mathrm{d}r\ge  \frac{1}{N}\Big(1+\mathrm{e}^{\Psi\big(\frac{N}{2}+1\big)+\gamma}\Big)\\
&+\frac{1}{c^{\frac{N+2}{N}}_{\epsilon}}\Big[\frac{1}{k!(\mu_{N}N)}\int_{0}^{-N\ln(\epsilon L_{\epsilon})}s^{\frac{N}{2}+1}\mathrm{e}^{-s}\mathrm{d}s+  O\Big(c^{\frac{N}{N+2}}_{\epsilon}(\epsilon L_{\epsilon})^{N}\Big)+O\Big(\frac{c^{\frac{N+2}{N}}_{\epsilon}}{a_{\epsilon}}\Big)\Big].
\end{aligned}
\end{equation}
Note that
\begin{equation}\nonumber
\int_{0}^{-N\ln(\epsilon L_{\epsilon})}s^{\frac{N}{2}+1}\mathrm{e}^{-s}\mathrm{d}s\to \Gamma(\frac{N}{2}+2)>0, \quad c^{\frac{N}{N+2}}_{\epsilon}(\epsilon L_{\epsilon})^{N}\rightarrow 0,\quad \frac{c^{\frac{N+2}{N}}_{\epsilon}}{a_{\epsilon}}\to 0
\end{equation}
as $\epsilon\to 0$. Thus, we obtain  
\begin{equation}\label{Sb-total-limpa}
\begin{aligned}
&\int_{0}^{1}r^{N-1}\mathrm{e}^{\mu_{N}|v_{\epsilon}|^{\frac{N+2}{N}}}\mathrm{d}r> \frac{1}{N}\Big(1+\mathrm{e}^{\Psi\big(\frac{N}{2}+1\big)+\gamma}\Big)
\end{aligned}
\end{equation}
for $\epsilon>0$ small enough.

Now, since $v_{\epsilon}\ge 1$ and $f>0$ on $(0,\epsilon L_{\epsilon})$,  from \eqref{Sb-part1} we have
\begin{equation}\label{Sb2-part1}
\int_{0}^{\epsilon L_{\epsilon}}r^{N-1}\mathrm{e}^{\mu_{N}|v_{\epsilon}|^{\frac{N+2}{N}+f(r)}}\mathrm{d}r\ge \int_{0}^{\epsilon L_{\epsilon}}r^{N-1}\mathrm{e}^{\mu_{N}|v_{\epsilon}|^{\frac{N+2}{N}}}\mathrm{d}r\ge \frac{1}{N}\mathrm{e}^{\Psi\big(\frac{N}{2}+1\big)+\gamma}+O\Big(\frac{1}{a_{\epsilon}}\Big).
\end{equation}
Let $0<a<1$. By using \eqref{limsup}, we get $(-
\ln r)^{f(r)}=\mathrm{e}^{f(r)\ln (-
\ln r)}\to 1$ as $r\to 0$. Thus, if $a$ is chosen small enough we have $$\sup\{\big(-N\ln r\big)^{\frac{N}{2}f(r)}\;\vert\; r\in (0,a]\}<\frac{3}{2}.$$ 
Hence, since $\int_{0}^{1}r^{N-1}(-N\ln r)^{\frac{N}{2}+1}\mathrm{d}r=\Gamma(\frac{N}{2}+2)/N$ we obtain
\begin{align*}
   0< \int_{0}^{a}r^{N-1}\big(-N\ln r\big)^{ \frac{N}{2}+1+\frac{N}{2}f(r)}\mathrm{d}r\le \frac{3}{2}\int_{0}^{a}r^{N-1}\big(-N\ln r\big)^{ \frac{N}{2}+1}\mathrm{d}r\le\frac{3\Gamma(\frac{N}{2}+2)}{2N}.
\end{align*}
In particular, 
\begin{equation}\label{f-oorigin}
   \lim_{\epsilon\to 0}\int_{0}^{\epsilon L_{\epsilon}}r^{N-1}\big(-N\ln r\big)^{ \frac{N}{2}+1+\frac{N}{2}f(r)}\mathrm{d}r=0.
\end{equation}
For $\epsilon>0$ small enough we have $0<\epsilon L_{\epsilon}<a$.  Thus,  $\mathrm{e}^{t}\ge 1+\frac{t^{k}}{k!},\;\; t\ge0$ yields (recall $k=N/2$)
\begin{equation}\label{Sb2-part2}
\begin{aligned}
&\int_{\epsilon L_{\epsilon}}^{1}r^{N-1}\mathrm{e}^{\mu_{N}|v_{\epsilon}|^{\frac{N+2}{N}+f(r)}}\mathrm{d}r\ge \frac{1}{N} +O\Big((\epsilon L_{\epsilon})^{N}\Big)+\frac{\mu^{k}_N}{k!}\int_{\epsilon L_{\epsilon}}^{1}r^{N-1}|v_{\epsilon}|^{k(\frac{N+2}{N}+f(r))}\mathrm{d}r\\
&=\frac{1}{N} +O\Big((\epsilon L_{\epsilon})^{N}\Big)+\frac{\mu^{k}_N}{k!}\int_{\epsilon L_{\epsilon}}^{1}r^{N-1}|v_{\epsilon}|^{ \frac{N}{2}+1+\frac{N}{2}f(r)}\mathrm{d}r\\
& \ge \frac{1}{N} +O\Big((\epsilon L_{\epsilon})^{N}\Big)+\frac{\mu^{k}_N}{k!}\int_{\epsilon L_{\epsilon}}^{a}r^{N-1}|v_{\epsilon}|^{ \frac{N}{2}+1+\frac{N}{2}f(r)}\mathrm{d}r\\
&=\frac{1}{N} +O\Big((\epsilon L_{\epsilon})^{N}\Big)+\frac{\mu^{k}_N}{k!}\int_{\epsilon L_{\epsilon}}^{a}r^{N-1}\Bigg(\frac{1}{\mu_{N} c^{\frac{2}{N}}_{\epsilon}}\Bigg)^{\frac{N}{2}+1+\frac{N}{2}f(r)}\Big(-N\ln r\Big)^{ \frac{N}{2}+1+\frac{N}{2}f(r)}\mathrm{d}r\\
& \ge \frac{1}{N} +O\Big((\epsilon L_{\epsilon})^{N}\Big)+\frac{\mu^{k}_N}{k!}\Bigg(\frac{1}{\mu_{N} c^{\frac{2}{N}}_{\epsilon}}\Bigg)^{ \frac{N}{2}+1+\frac{N}{2}m_a}\int_{\epsilon L_{\epsilon}}^{a}r^{N-1}\big(-N\ln r\big)^{ \frac{N}{2}+1+\frac{N}{2}f(r)}\mathrm{d}r,
\end{aligned}
\end{equation}
where $0<m_{a}=\max\{f(r)\; \vert\; r\in [0, a]\}<\infty$. Moreover, from 
\eqref{f-oorigin} 
\begin{equation}\label{inc-gamma}
\begin{aligned}
\int_{\epsilon L_{\epsilon}}^{a}r^{N-1}\big(-N\ln r\big)^{ \frac{N}{2}+1+\frac{N}{2}f(r)}\mathrm{d}r&=\int_{0}^{a}r^{N-1}\big(-N\ln r\big)^{ \frac{N}{2}+1+\frac{N}{2}f(r)}\mathrm{d}r\\
& \quad -\int_{0}^{\epsilon L_{\epsilon}}r^{N-1}\big(-N\ln r\big)^{ \frac{N}{2}+1+\frac{N}{2}f(r)}\mathrm{d}r\\
& =\int_{0}^{a}r^{N-1}\big(-N\ln r\big)^{ \frac{N}{2}+1+\frac{N}{2}f(r)}\mathrm{d}r+o(1).
\end{aligned}
\end{equation}
 Combining \eqref{Sb2-part1},\eqref{Sb2-part2} and \eqref{inc-gamma} we can write
\begin{equation}\label{Sb2-final}
\begin{aligned}
&\int_{0}^{1}r^{N-1}\mathrm{e}^{\mu_{N}|v_{\epsilon}|^{\frac{N+2}{N}+f(r)}}\mathrm{d}r\ge \frac{1}{N}\Big(1+\mathrm{e}^{\Psi\big(\frac{N}{2}+1\big)+\gamma}\Big)
 \\
&+\frac{1}{\mu^{1+\frac{N}{2}m_a}_Nc^{1+\frac{2}{N}+m_a}_{\epsilon}k!}\left[\int_{0}^{a}r^{N-1}\big(-N\ln r\big)^{ \frac{N}{2}+1+\frac{N}{2}f(r)}\mathrm{d}r\right.\\
&\left.\quad\quad+O\Big(\frac{c^{1+\frac{2}{N}+m_a}_{\epsilon}}{a_{\epsilon}}\Big)+O\Big((\epsilon L_{\epsilon})^{N}c^{1+\frac{2}{N}+m_a}_{\epsilon}\Big)+o(1)\right].
\end{aligned}
\end{equation}
From \eqref{Sb2-final} it remains to ensure that there exists $0<a<1$ (small enough) such that 
\begin{equation}\label{choice-aaa}
\tau:=1+\frac{2}{N}+m_a\;\Rightarrow\;\frac{c^{\tau}_{\epsilon}}{a_{\epsilon}}\to 0\quad\mbox{and}\quad(\epsilon L_{\epsilon})^{N}c^{\tau}_{\epsilon}\to 0,\;\;\mbox{as}\;\; \epsilon\to 0. 
\end{equation}
However, from \eqref{c} we have
\begin{equation}\nonumber
\begin{aligned}
c^{\tau}_{\epsilon}=\left[\frac{N}{2\mu_{N}}\left(\ln(1+a_{\epsilon})-\Psi\Big(\frac{N}{2}+1\Big)-\gamma+O\left(\frac{1}{1+a_{\epsilon}}\right)-2\ln \epsilon L_{\epsilon}\right)\right]^{\frac{\tau N}{N+2}}.
\end{aligned}
\end{equation}
Consequently,
\begin{equation}\nonumber
\begin{aligned}
\frac{c^{\tau}_{\epsilon}}{a_{\epsilon}}=\left[\frac{N}{2\mu_{N}}\frac{1}{a^{\frac{N+2}{\tau N}}_{\epsilon}}\left(\ln(1+a_{\epsilon})-\Psi\Big(\frac{N}{2}+1\Big)-\gamma+O\left(\frac{1}{1+a_{\epsilon}}\right)-2\ln \epsilon L_{\epsilon}\right)\right]^{\frac{\tau N}{N+2}}.
\end{aligned}
\end{equation}
We recall that $a_{\epsilon}=a_{N}L^{2}_{\epsilon}$, with $L_{\epsilon}=-\ln\epsilon$. Thus, $\ln(1+a_{\epsilon})/a_{\epsilon}^{\eta}\to 0$ as $\epsilon\to 0$, for any $\eta>0$. Hence,  to conclude the first part in \eqref{choice-aaa}, it is sufficient to choose $0<a<1$ such that  
\begin{equation}\nonumber
\frac{x-\ln x}{x^{\frac{2(N+2)}{\tau N}}}\to 0\quad\mbox{as}\quad x\to +\infty
\end{equation}
which holds provided that $a$ is chosen such that $0<m_a<1+\frac{2}{N}$ that is possible due to $f(0)=0$ and $f>0$ on $(0,1)$. In addition, 
\begin{equation}\nonumber
\begin{aligned}
(\epsilon L_{\epsilon})^{N} c^{\tau}_{\epsilon}=\left[\frac{N}{2\mu_{N}}(\epsilon L_{\epsilon})^{\frac{N+2}{\tau}} \left(\ln(1+a_{\epsilon})-\Psi\Big(\frac{N}{2}+1\Big)-\gamma+O\left(\frac{1}{1+a_{\epsilon}}\right)-2\ln \epsilon L_{\epsilon}\right)\right]^{\frac{\tau N}{N+2}}.
\end{aligned}
\end{equation}
Note that $(\epsilon L_{\epsilon})^{\eta}\ln (\epsilon L_{\epsilon})\to 0$, as $\epsilon\rightarrow 0$, for any $\eta>0$. Thus,  we only need to choose $a$ such that
\begin{equation}\nonumber
\begin{aligned}
&\lim_{\epsilon\to 0}(\epsilon L_{\epsilon})^{\frac{N+2}{\tau}}\ln (1+a_{\epsilon})=\lim_{x\rightarrow\infty} \mathrm{e}^{-\frac{(N+2)x}{\tau}}x^{\frac{N+2}{\tau}}\ln( 1+a_{N}x^2)=0,
\end{aligned}
\end{equation}
which holds for any $m_a>0$.
\end{proof}
\begin{lemma}[concentration-compactness] \label{cc-a} Suppose that $f$  satisfies  \eqref{limsup}. Let $(v_j)\subset X_{1}$ be a sequence such that
$\|v_{j}\|_{X_1}=1$ and $v_{j}\rightharpoonup v\not\equiv 0$ weakly in $X_1.$ Set 
\begin{equation}
P=\left\{\begin{aligned}
&\left(1-\|v\|^{k+1}_{X_1}\right)^{-\frac{2}{N}},\;\; 
&\mbox{if}&\;\; \|v\|_{X_1}<1\\
&\infty, \;\; &\mbox{if}&\;\; \|v\|_{X_1}=1.\\
\end{aligned}\right.
\end{equation}
Then, for any $1<q<P$ we have
\begin{equation}\label{2uniformboundedness}
\limsup_{j}\int_{0}^{1}r^{N-1}\mathrm{e}^{q\mu_{N}|v_j|^{\frac{N+2}{N}+f(r)}}\mathrm{d}r<\infty.
\end{equation}
\end{lemma}
\begin{proof} From \eqref{r-estimate} we have
\begin{equation}\nonumber
\begin{aligned}
|v_j(r)|^{f(r)} \le \Big(-\frac{N}{\mu_{N}}\ln r\Big)^{\frac{N}{N+2}f(r)}, \quad r>0.
\end{aligned}
\end{equation}
From  \eqref{limsup},  we get $|\ln r|^{f(r)}\to 1$ as $r\to 0^{+}$. Then, for any $1<q<P$, we can pick $\rho>0$ small enough and $q<\overline{q}<P$ such that 
\begin{equation}\nonumber
\begin{aligned}
q\mu_{N}|v_j(r)|^{\frac{N+2}{N}+f(r)}&=q\mu_{N}|v_j(r)|^{f(r)} |v_j(r)|^{\frac{N+2}{N}}\\
&\le \Big(-\frac{N}{\mu_{N}}\ln r\Big)^{\frac{N}{N+2}f(r)}q\mu_{N}|v_j(r)|^{\frac{N+2}{N}}\\
&\le \overline{q}\mu_{N}|v_j(r)|^{\frac{N+2}{N}},
\end{aligned}
\end{equation}
for all $r\in (0,\rho)$. Thus,  from \cite[Theorem~1]{JJ2014CV}, we get 
\begin{equation}\label{2limisuporigin}
\limsup_{j\rightarrow\infty}\int_{0}^{\rho}r^{N-1}\mathrm{e}^{q\mu_{N}|v_j|^{\frac{N+2}{N}+f(r)}}\ud r\le \limsup_{j\rightarrow\infty}\int_{0}^{1}r^{N-1}\mathrm{e}^{\overline{q}\mu_{N}|v_j|^{\frac{N+2}{N}}}\ud r<+\infty.
\end{equation} 
On the other hand, since $f>0$ and $-N\ln r\to 0$ as $r\to 1^{-}$ we can take 
$$
\mathcal{C}=\sup\{\big(\frac{N}{\mu_N}|\ln r|\big)^{1+\frac{N}{N+2}f(r)}\;\vert\; r\in [\rho,1)\}<\infty.
$$
Thus, we from \eqref{r-estimate}  can write
\begin{equation}\label{2Ltfar0}
\begin{aligned}
\int_{\rho}^{1}r^{N-1}\mathrm{e}^{q\mu_{N}|v_j|^{\frac{N+2}{N}+f(r)}} \ud r\le \mathrm{e}^{q\mu_{N}\mathcal{C}}\int_{0}^{1}r^{N-1}\ud r. 
\end{aligned}
\end{equation}
By using \eqref{2limisuporigin} and \eqref{2Ltfar0} we obtain \eqref{2uniformboundedness}.
\end{proof}
\subsection{Attainability for $\mathcal{MT}_{N}(f)$}
Next, we are proving that  the  supremum $\mathcal{MT}_{N}(f)$ is attained provided that $f$  satisfies  \eqref{limsup}, with $\sigma>2$. Let $(v_j)\subset X_1$ be a maximizing sequence for $\mathcal{MT}_{N}(f)$, that is, $\|v_j\|_{X_1}=1$, 
 \begin{equation}
\mathcal{MT}_{N}(f)=\lim_{j\rightarrow\infty}\int_{0}^{1}r^{N-1}\mathrm{e}^{\mu_{N}|v_j|^{\frac{N+2}{N}+f(r)}}\ud r.
\end{equation}
Without loss of generality, we can assume that
\begin{equation}\label{weak-limitmax}
v_{j}\rightharpoonup v\;\;\mbox{weakly in}\;\; X_{1}\;\; \mbox{and}\;\; v_j(r)\rightarrow v(r)\;\;\mbox{a.e in}\; (0,1).
\end{equation}
If  $v\not\equiv 0$, from  the Vitali convergence theorem,  Lemma~\ref{cc-a}  and  \eqref{weak-limitmax} we ensure that $v$ is an extremal function for $\mathcal{MT}_{N}(f)$. Further, the following result, together with \eqref{lower-bound}, excludes  the case  $v\equiv0$ and ensures the attainability of $\mathcal{MT}_{N}(f)$.
\begin{lemma} Assume that $v\equiv 0$ in \eqref{weak-limitmax}. Then either 
\begin{enumerate}
\item [$(a)$] $(v_j)$ is  a NCS at origin and the inequality
\begin{equation}\label{ineq-false}
\mathcal{MT}_{N}(f)\le \frac{1}{N}\Big(1+\mathrm{e}^{\gamma+\Psi\big(\frac{N}{2}+1\big)}\Big)
\end{equation}
holds or
\item [$(b)$] $(v_j)$ is not a NCS at origin and the identity
\begin{equation}\label{equal-false}
\mathcal{MT}_{N}(f)=\frac{1}{N}
\end{equation}
holds.
\end{enumerate}
\end{lemma}
\begin{proof}
For any $\rho>0$,  \eqref{r-estimate} and Lebesgue’s dominated convergence theorem yield
\begin{equation}\label{2-partin}
\begin{aligned}
\int_{\rho}^{1}r^{N-1}\mathrm{e}^{\mu_{N}|v_j|^{\frac{N+2}{N}+f(r)}}\ud r& =\int_{\rho}^{1}r^{N-1}\ud r+\int_{\rho}^{1}r^{N-1}\Big(\mathrm{e}^{\mu_{N}|v_j|^{\frac{N+2}{N}+f(r)}}-1\Big)\ud r\\
&=\int_{\rho}^{1}r^{N-1}\ud r +o(1)\\
&= \int_{\rho}^{1}r^{N-1}\mathrm{e}^{\mu_{N}|v_j|^{\frac{N+2}{N}}}\ud r +o(1).
\end{aligned}
\end{equation}
Also, 
\begin{equation}\label{2-partout}
\begin{aligned}
\int_{0}^{\rho}r^{N-1}\mathrm{e}^{\mu_{N}|v_j|^{\frac{N+2}{N}+f(r)}}\ud r&=\int_{0}^{\rho}r^{N-1}\Big(\mathrm{e}^{\mu_{N}|v_j|^{\frac{N+2}{N}}\big(|v_j|^{f(r)}-1\big)}-1\Big)\mathrm{e}^{\mu_{N}|v_j|^{\frac{N+2}{N}}}\ud r\\
& +\int_{0}^{\rho}r^{N-1}\mathrm{e}^{\mu_{N}|v_j|^{\frac{N+2}{N}}}\ud r.
\end{aligned}
\end{equation}
Combining \eqref{2-partin} with \eqref{2-partout}, we can write
\begin{equation}\label{2-partinout}
\begin{aligned}
\int_{0}^{1}r^{N-1}\mathrm{e}^{\mu_{N}|v_j|^{\frac{N+2}{N}+f(r)}}\ud r&=\int_{0}^{\rho}r^{N-1}\Big(\mathrm{e}^{\mu_{N}|v_j|^{\frac{N+2}{N}}\big(|v_j|^{f(r)}-1\big)}-1\Big)\mathrm{e}^{\mu_{N}|v_j|^{\frac{N+2}{N}}}\ud r\\
& +\int_{0}^{1}r^{N-1}\mathrm{e}^{\mu_{N}|v_j|^{\frac{N+2}{N}}}\ud r +o(1).
\end{aligned}
\end{equation}
Now, we are estimate the integral $\int_{0}^{\rho}$ on the right hand side of \eqref{2-partinout} for $\rho>0$ small enough. By using \eqref{r-estimate} and \eqref{limsup}, for $\rho>0$ (small)
\begin{equation}\nonumber
\begin{aligned}
\mathrm{e}^{\mu_{N}|v_j|^{\frac{N+2}{N}}\big(\big(|v_j|^{f(r)}-1\big)}& \le
\mathrm{e}^{(-N\ln r)\big((-\frac{N}{\mu_{N}}\ln r)^{\frac{N}{N+2}f(r)}-1\big)}\\
& \le \mathrm{e}^{(-N \ln r)\big( \big(-\frac{N}{\mu_{N}}\ln r\big)^{\frac{N}{N+2}\frac{c}{(-\ln r )^{\sigma}}}-1\big)}, 
\end{aligned}
\end{equation}
for some $c>0$ and for any $r\in (0, \rho)$. Note that 
\begin{equation}\nonumber
  \begin{aligned}
      \big(-\frac{N}{\mu_{N}}\ln r\big)^{\frac{N}{N+2}\frac{c}{(-\ln r )^{\sigma}}}&=\mathrm{e}^{\frac{N}{N+2}\frac{c}{(-\ln r )^{\sigma}}\ln \big(-\frac{N}{\mu_{N}}\ln r\big)}\\
      &= 1+ \frac{N}{N+2}\frac{c}{(-\ln r )^{\sigma}}\ln \big(-\frac{N}{\mu_{N}}\ln r\big)+o\Big(\frac{1}{(-\ln r)^{\sigma}}\ln \big(-\frac{N}{\mu_{N}}\ln r\big)\Big)_{\searrow 0}
  \end{aligned}  
\end{equation}
and thus 
\begin{equation}\nonumber
  \begin{aligned}
      (-N\ln r)\Big(\big(-\frac{N}{\mu_{N}}\ln r\big)^{\frac{N}{N+2}\frac{c}{(-\ln r )^{\sigma}}}-1\Big)&=\frac{N^2}{N+2}\frac{c}{(-\ln r )^{\sigma-1}}\ln \big(-\frac{N}{\mu_{N}}\ln r\big)\big[1+o(1)_{\searrow 0}\big].
  \end{aligned}  
\end{equation}
Therefore,  by using  $\mathrm{e}^{x}-1=x+o(x)$, as $x\to 0$,  we can write 
\begin{equation}\label{H-assym}
\begin{aligned}
    H(r) & := \mathrm{e}^{(-N \ln r)\big( \big(-\frac{N}{\mu_{N}}\ln r\big)^{\frac{N}{N+2}\frac{c}{(-\ln r )^{\sigma}}}-1\big)}-1\\
    & =\frac{cN^{2}}{(N+2)(-\ln r)^{\sigma-1}}\ln\Big(-\frac{N}{\mu_{N}}\ln r\Big)[1+o(1)_{r \searrow 0}].
    \end{aligned}
\end{equation}
Hence, using \eqref{r-estimate}
\begin{equation}\nonumber
\begin{aligned}
 \int_{0}^{\rho}r^{N-1}\Big(\mathrm{e}^{\mu_{N}|v_j|^{\frac{N+2}{N}}\big(|v_j|^{f(r)}-1\big)}-1\Big)\mathrm{e}^{\mu_{N}|v_j|^{\frac{N+2}{N}}}\ud r
 & \le \int_{0}^{\rho}r^{-1}\Big(\mathrm{e}^{-N\ln r\big((-\frac{N}{\mu_{N}}\ln r)^{\frac{N}{N+2}f(r)}-1\big)}-1\Big)\ud r\\
 &\le  \int_{0}^{\rho}r^{-1}H(r)\ud r.
\end{aligned}
\end{equation}
But, from \eqref{H-assym} and using that $\sigma>2$ we have 
\begin{align*}
     \int_{0}^{\rho}r^{-1}H(r)\mathrm{d}r &\le \frac{2cN^{2}}{N+2}\int_{0}^{\rho}r^{-1}\frac{1}{(-\ln r)^{\sigma-1}}\ln\Big(-\frac{N}{\mu_{N}}\ln r\Big)\mathrm{d}r\\
    &= \frac{2cN^{2}}{N+2}\int_{-\ln \rho}^{\infty}t^{1-\sigma}\ln\Big(\frac{Nt}{\mu_{N}}\Big)\mathrm{d}t<\infty.
\end{align*}
Hence
\begin{equation}\nonumber
 \lim_{\rho\to 0} \int_{0}^{\rho}r^{N-1}\Big(\mathrm{e}^{\mu_{N}|v_j|^{\frac{N+2}{N}}\big(|v_j|^{f(r)}-1\big)}-1\Big)\mathrm{e}^{\mu_{N}|v_j|^{\frac{N+2}{N}}}\ud r\le  \lim_{\rho\to 0} \int_{0}^{\rho}r^{-1}H(r)\ud r=0.
\end{equation}
Thus, if $(v_j)$ is  a NCS at origin, letting $\rho \to 0$ and $j\to\infty$ in \eqref{2-partinout}  and using Lemma~\ref{S-lemma} we obtain \eqref{ineq-false}.

Assuming that $(v_j)$  is not a NCS, we will  prove \eqref{equal-false}. Indeed, there exists $r_0>0$ such that
\begin{equation}\nonumber
\liminf_{j}c_{N}\int_{r_0}^{1}r^{N-k}|v^{\prime}_{j}|^{k+1}\ud r>0
\end{equation}
and, from $\|v_j\|_{X_1}=1$, we have
\begin{equation}\label{<1}
\limsup_{j}c_{N}\int_{0}^{r_0}r^{N-k}|v^{\prime}_{j}|^{k+1}\ud r\le \delta<1.
\end{equation}
Let us  take the auxiliary function
$\xi\in C^{1}([0,1]),$ with $0\le \xi\le 1$ such that
\begin{equation}\nonumber
\xi(r)=\left\{\begin{aligned}
&1,\;\;\mbox{if}\;\; r\in [0,r_0/2]\\
&0,\;\;\mbox{if}\;\; r\in [r_0, 1].
\end{aligned}\right.
\end{equation}
We have 
\begin{equation}\label{26ago03}
\begin{aligned}
   \|\xi v_j\|_{X_1} &= \|c_N\xi v^{\prime}_{j}+c_N\xi^{\prime}v_j\|_{L^{k+1}_{N-k}} \\
   & \le  \|c_N\xi v^{\prime}_j\|_{L^{k+1}_{N-k}}+\|c_N\xi^{\prime}v_j\|_{L^{k+1}_{N-k}}\\
   & \le \left(c_N\int_{0}^{r_0}r^{N-k}|v^{\prime}_{j}|^{k+1}\ud r\right)^{\frac{1}{k+1}}+\|c_N\xi^{\prime}v_j\|_{L^{k+1}_{N-k}}.
\end{aligned}
\end{equation}
From the compact embedding \eqref{EbeddingsTM} and \eqref{weak-limitmax}, we get
$v_j\rightarrow0$ in $L^{k+1}_{N-k}$. Then, \eqref{<1} and \eqref{26ago03} yield
\begin{equation}\label{norma<1}
\begin{aligned}
    \|\xi v_j\|_{X_1}<1,
\end{aligned}
\end{equation}
for large enough $j\in\mathbb{N}$.  From \eqref{norma<1},  there is $q>1$ such that $q \|\xi v_j\|^{\frac{N+2}{N}+f(r)}_{X_1}\le 1$, for $0\le r\le 1$. Thus, from  \eqref{r-estimate} and \eqref{supii}
\begin{equation}\nonumber
\begin{aligned}
 \int_{0}^{1}r^{N-1}\mathrm{e}^{q\mu_{N}|v_j|^{\frac{N+2}{N}+f(r)}}\ud r & \le \int_{0}^{\frac{r_0}{2}}r^{N-1}\mathrm{e}^{q\mu_{N}|v_j|^{\frac{N+2}{N}+f(r)}}\ud r +c_2 \\
 & = \int_{0}^{\frac{r_0}{2}}r^{N-1}\mathrm{e}^{q \|\xi v_j\|^{\frac{N+2}{N}+f(r)}_{X_1}\mu_{N}\big|\frac{\xi v_j}{ \|\xi v_j\|_{X_1}}\big|^{\frac{N+2}{N}+f(r)}}\ud r+ c_2\\
& \le  \int_{0}^{1}r^{N-1}\mathrm{e}^{\mu_{N}\big|\frac{\xi v_j}{ \|\xi v_j\|_{X_1}}\big|^{\frac{N+2}{N}+f(r)}}\ud r+ c_2\\
& \le c_1 +c_2.
\end{aligned}
\end{equation}
Hence, the Vitali convergence theorem yields  $\mathcal{MT}_{N}(f)=1/N$.
\end{proof}
\subsection{Regularity for auxiliary supercritical maximizers} 
Next, we prove that  if $v\in X_1$ is a maximizer for  $\mathcal{MT}_{N}(f)$ then $v\in C^{2}(0,1]\cap C^{1}[0,1]$. Firstly, inspired by \cite[Lemma~6]{MZ2020}, we prove the following:
\begin{lemma}\label{0-beha}  Let  $v\in X^{k+1}_{1}\;(k=N/2)$ be a non-increasing function.  Then,  for any  $\vartheta,\gamma_0>0$ it holds
\begin{equation}\label{MZ-beharior}
\lim_{r\to 0^+} r^{\vartheta}|v(r)|^{\frac{2}{N}}\mathrm{e}^{\gamma_0|v(r)|^{\frac{N+2}{N}}}=0.
\end{equation}

\end{lemma}
\begin{proof}
There exists $m>1$ and $C>0$ depending only on $m, N$ and $\vartheta$ such that 
\begin{equation}\label{emb1}
\|v^{\prime}\|_{L^{m}_{\vartheta}}\le C\|v\|_{X_1}.
\end{equation}
Indeed, taking $m, q>1$   such that
$(\vartheta+1)q>k+1\ge mq$, the  H\"{o}lder inequality  yields 
\begin{equation}\nonumber
\begin{aligned}
\|v^{\prime}\|_{L^{mq}_{N-k}} &\le \Big(\int_{0}^{1}r^{N-k}\ud r\Big)^{\frac{1}{mq}-\frac{1}{k+1}}\|v^{\prime}\|_{L^{k+1}_{N-k}}.
\end{aligned}
\end{equation}
In addition,  using the above estimate and our choices for $m$ and $q$, we can write  (recall $k=N/2$ and $(\vartheta-\frac{k}{q})\frac{q}{q-1}>-1$)
 \begin{eqnarray*}
\int_{0}^1 r^{\vartheta}|v^{\prime}|^{m}\ud r&=&\int_{0}^1 
|v^{\prime}|^{m}r^{\frac{N-k}{q}}r^{\vartheta-\frac{k}{q}}\ud
r\\
&\le& \left(\int_{0}^{1}|v^{\prime}|^{mq}r^{N-k}\ud
r\right)^{\frac{1}{q}}\left(\int_{0}^{1}r^{(\vartheta-\frac{k}{q})\frac{q}{q-1}}\ud
r\right)^{\frac{q-1}{q}}\\
&\le & C\|v\|^{m}_{X_1},
 \end{eqnarray*}
which proves \eqref{emb1}.\\
Set 
$$
\varphi(r)=|v(r)|^{\frac{2}{N}}\, \mathrm{e}^{\gamma_0|v(r)|^{\frac{N+2}{N}}}.
$$
It remains to show that
\begin{equation}
\lim_{r\rightarrow 0^{+}}r^{\vartheta}\varphi(r)=0.
\end{equation}
Since $v$ is non-increasing function,  without loss of generality we can assume $v>0$ in $(0,1/2)$ and $\lim_{r\rightarrow 0^{+}}v(r)=+\infty$. Hence, there exists $C=C(N)$ such that  (use that $1/v$ is bounded near $0$)
\begin{eqnarray}
|\varphi^{\prime}(r)|&\le &\frac{2}{N}|v(r)|^{\frac{2-N}{N}}
|v^{\prime}(r)| \mathrm{e}^{\gamma_0 |v(r)|^{\frac{N+2}{N}}}+\gamma_0\frac{N+2}{N}|v(r)|^{\frac{4}{N}}|v^{\prime}(r)|
\mathrm{e}^{\gamma_0 |v(r)|^{\frac{N+2}{N}}}\nonumber\\
&\le & C\left(
|v^{\prime}(r)| \mathrm{e}^{\gamma_0 |v(r)|^{\frac{N+2}{N}}}+|v(r)|^{\frac{4}{N}}|v^{\prime}(r)|
\mathrm{e}^{\gamma_0 |v(r)|^{\frac{N+2}{N}}}\right)\label{11fev201301},
\end{eqnarray}
for any $r\in(0,1/2)$. Now, we claim  that 
\begin{equation}\label{phi-claim}
\varphi^{\prime}\in L^{1}_{\vartheta}(1,1/2).
\end{equation}
Certainly, for $m>1$ such as in \eqref{emb1}, $p>1$  and $q>1$ such that
$$
\frac{4}{qN}+\frac{1}{m}+\frac{1}{p}=1
$$
the H\"{o}lder inequality yields
\begin{equation}\nonumber
\int_{0}^{1}r^{\vartheta}|v|^{\frac{4}{N}}|v^{\prime}|
\mathrm{e}^{\gamma_0|v|^{\frac{N+2}{N}}}\mathrm{d}r\le\left(\int_{0}^{1}r^{\vartheta}|v|^{q}\mathrm{d}r\right)^{\frac{4}{qN}}\left(\int_{0}^{1}r^{\vartheta}|v^{\prime}|^{m}\mathrm{d}r\right)^{\frac{1}{m}}\left(\int_{0}^{1}r^{\vartheta}\mathrm{e}^{\gamma_0 p|v|^{\frac{N+2}{N}}}\mathrm{d}r\right)^{\frac{1}{p}}.
\end{equation}
The above estimate combined with \eqref{emb1}, the embedding \eqref{EbeddingsTM} and the weighted Trudinger-Moser i\-ne\-qua\-li\-ty \cite{JJ2012} gives
\begin{equation}\label{phi-part1}
|v|^{\frac{4}{N}}|v^{\prime}|
\mathrm{e}^{\gamma_0|v|^{\frac{N+2}{N}}}\in L^{1}_{\vartheta}(0,1).
\end{equation}
Analogously, taking $m>1$ such as in
\eqref{emb1}, the
H\"{o}lder inequality  implies
\begin{equation}\nonumber
\int_{0}^{1}r^{\vartheta}|v^{\prime}|
\mathrm{e}^{\gamma_0|v|^{\frac{N+2}{N}}}\ud r\le \left(\int_{0}^{1}r^{\vartheta}|v^{\prime}|^{m}\mathrm{d}r\right)^{\frac{1}{m}}\left(\int_{0}^{1}r^{\vartheta}\mathrm{e}^{\gamma_0\frac{m}{m-1}|v|^{\frac{N+2}{N}}}\mathrm{d}r\right)^{\frac{m-1}{m}}.
\end{equation}
Thus, also from  \cite{JJ2012} we have
\begin{equation}\label{phi-part2}
|v^{\prime}|
\mathrm{e}^{\gamma_0|v|^{\frac{N+2}{N}}}\in L^{1}_{\vartheta}(0,1).
\end{equation}
Hence, combining  \eqref{11fev201301}, \eqref{phi-part1} and \eqref{phi-part2}, we prove \eqref{phi-claim}. For $0<r<s<1/2$ we have 
\begin{equation}\nonumber
|\varphi(r)|\le \varphi\Big(\frac{1}{2}\Big)+\int_{r}^{s}|\varphi^{\prime}(\tau)|\mathrm{d}\tau+\int_{s}^{\frac{1}{2}}|\varphi^{\prime}(\tau)|\mathrm{d}\tau.
\end{equation}
Therefore, 
\begin{equation}\label{phi-calculusF}
r^{\vartheta}|\varphi(r)|\le r^{\vartheta}\varphi\Big(\frac{1}{2}\Big)+\int_{r}^{s}\tau^{\vartheta}|\varphi^{\prime}(\tau)|\mathrm{d}\tau+\frac{r^{\vartheta}}{s^{\vartheta}}\int_{s}^{\frac{1}{2}}\tau^{\vartheta}|\varphi^{\prime}(\tau)|\mathrm{d}\tau.
\end{equation}
From \eqref{phi-claim}, for each $\epsilon>0$  there exists $s>0$ (small) such that
$$\int_{r}^{s}\tau^{\vartheta}|\varphi^{\prime}(\tau)|\mathrm{d}\tau<\frac{\epsilon}{3}.
$$
Also, there exists $0<\delta< s$ such that
$$\delta^{\vartheta}|\varphi\Big(\frac{1}{2}\Big)|<\frac{\epsilon}{3}\;\;\;\mbox{and}\;\;\;
\frac{\delta^{\vartheta}}{s^{\vartheta}}\int_{s}^{\frac{1}{2}}\tau^{\vartheta}|\varphi^{\prime}(\tau)|\mathrm{d}\tau<\frac{\epsilon}{3}.$$
Hence, using \eqref{phi-calculusF}, we conclude that  
$$
r^{\vartheta}|\varphi(r)|<\epsilon, \;\;\;\mbox{for}\;\;\; 0<r<\delta
$$
which completes the proof.
\end{proof}
\begin{lemma} \label{lemma-F0} Suppose that $f$  satisfies \eqref{limsup}. Then,
for each $v\in X_1$, there is $c=c(v)>0$ such that
\begin{equation}\nonumber
    \begin{aligned}
        \sup_{r\in (0,1)}\Big(\frac{N+2}{N}+f(r)\Big)|v(r)|^{f(r)}\le c.
    \end{aligned}
\end{equation} 
Moreover, the constant $c$ can be choose independent of $v\in X_1$ if $\|v\|_{X_1}=1$.
\end{lemma}
\begin{proof}
For each $v\in X_1$, from  \eqref{r-estimate}
\begin{equation}\label{well-a0log}
|v(r)|^{f(r)} \le \Big[\Big(\frac{N}{\mu_{N}}\Big)^{\frac{N}{N+2}}\|v\|_{X_1}\left(-\ln r\right)^{\frac{N}{N+2}}\Big]^{f(r)}.
\end{equation}
Then,
\begin{equation}\label{well-wells}
    \begin{aligned}
       \Big(\frac{N+2}{N}+f(r)\Big)|v(r)|^{f(r)} & =\frac{\frac{N+2}{N}+f(r)}{\mathrm{e}^{f(r)}} \mathrm{e}^{f(r)}|v(r)|^{{f(r)}}\\
        & \le \frac{\frac{N+2}{N}+f(r)}{\mathrm{e}^{f(r)}} \mathrm{e}^{f(r)\big[\ln \big((\frac{N}{\mu_{N}})^{\frac{N}{N+2}}\|v\|_{X_1}\left(-\ln r\right)^{\frac{N}{N+2}}\big)+1\big]}.
    \end{aligned}
\end{equation}
Since $(-\ln r)\to 0$ as $r\to 1$, we get $$\ln \big((\frac{N}{\mu_{N}})^{\frac{N}{N+2}}\|v\|_{X_1}\left(-\ln r\right)^{\frac{N}{N+2}}\big)+1<0$$ for $r$ close to $1$. Hence, by using $0 \le {(\frac{N+2}{N}+x})\mathrm{e}^{-x}\le \frac{N+2}{N}$ for $x\ge0$ we obtain 
\begin{equation}\nonumber
 \limsup_{r\to 1}   \frac{\frac{N+2}{N}+f(r)}{\mathrm{e}^{f(r)}} \mathrm{e}^{f(r)\big(\ln \big(\big(\frac{N}{\mu_{N}}\big)^{\frac{N}{N+2}}\|v\|_{X_1}\left(-\ln r\right)^{\frac{N}{N+2}}\big)+1\big)}\le\frac{N+2}{N}.
\end{equation}
On the other hand, from \eqref{limsup} and $f(0)=0$
\begin{equation}\nonumber
 \lim_{r\to 0}   \frac{\frac{N+2}{N}+f(r)}{\mathrm{e}^{f(r)}} \mathrm{e}^{f(r)\big[\ln \big(\big(\frac{N}{\mu_{N}}\big)^{\frac{N}{N+2}}\|v\|_{X_1}\left(-\ln r\right)^{\frac{N}{N+2}}\big)+1\big]}=\frac{N+2}{N}.
\end{equation}
Thus, from \eqref{well-wells} we may pick 
\begin{equation*}
    c=c(v)=\sup_{r\in (0,1)} \frac{\frac{N+2}{N}+f(r)}{\mathrm{e}^{f(r)}} \mathrm{e}^{f(r)\big[\ln \big(\big(\frac{N}{\mu_{N}}\big)^{\frac{N}{N+2}}\|v\|_{X_1}\left(-\ln r\right)^{\frac{N}{N+2}}\big)+1\big]}<\infty.
\end{equation*}
If $\|v\|_{X_1}=1$, from \eqref{well-a0log} and \eqref{well-wells}, we can take 
\begin{equation*}
    c=\sup_{r\in (0,1)} \frac{\frac{N+2}{N}+f(r)}{\mathrm{e}^{f(r)}} \mathrm{e}^{f(r)\big[\ln \big(\big(\frac{N}{\mu_{N}}\big)^{\frac{N}{N+2}}\left(-\ln r\right)^{\frac{N}{N+2}}\big)+1\big]}<\infty.
\end{equation*}
\end{proof}
\begin{lemma}\label{vc2} Let $v\in X_1$ be non-negative maximizer for $\mathcal{MT}_{N}(f)$. Then, $v$ is a decreasing function and $v\in C^{2}[0,1]$.
\begin{proof}
Let $F: X_1\to \mathbb{R}$ be the supercritical Trudinger-Moser type functional
\begin{equation*}
    F(v)=\int_{0}^{1} r^{N-1}\mathrm{e}^{\mu_{N}|v|^{\frac{N+2}{N}+f(r)}}\ud r.
\end{equation*}
From Proposition~\ref{prop1} and Lemma~\ref{lemma-F0}, we can see that $F\in C^{1}(X_1,\mathbb{R})$ with
\begin{equation}\nonumber
  F^{\prime}(v).h =   \mu_{N}\int_{0}^{1}r^{N-1} \big(\frac{N+2}{N}+f(r)\big)|v|^{\frac{2}{N}+f(r)-1}\mathrm{e}^{\mu_{N}|v|^{\frac{N+2}{N}+f(r)}} v h.
\end{equation}
Then, if $v\in X_1$ is a non-negative maximizer for $\mathcal{MT}_{N}(f)$, the Lagrange multipliers theorem yields
\begin{equation}\label{weaksolution}
c_{N}\int_{0}^{1}r^{N-k}|v^{\prime}|^{k-1}v^{\prime}h^{\prime}\,\mathrm{d}
r=\lambda\int_{0}^{1}r^{N-1} \big(\frac{N+2}{N}+f(r)\big)|v|^{\frac{2}{N}+f(r)}\mathrm{e}^{\mu_{N}|v|^{\frac{N+2}{N}+f(r)}}h\,\mathrm{d}r,\;\; \forall\;\; h\in X_1
\end{equation}
where
\begin{equation}\label{constants}
\lambda=\Big(\mu_N\int_{0}^{1}r^{N-1}\big(\frac{N+2}{N}+f(r)\big)|v|^{\frac{N+2}{N}+f(r)}\mathrm{e}^{\mu_{N}|v|^{\frac{N+2}{N}+f(r)}}\,\mathrm{d}r\Big)^{-1}.
\end{equation}
Following  \cite{Clement-deFigueiredo-Mitidieri},   for each $r\in(0,1)$ and $\rho>0$ let
$h_{\rho}\in X_1$ be given by
\begin{equation}\label{keyfunction}
h_{\rho}(s)=\left\{\begin{aligned}&\;1 &\mbox{if}\quad &0\le s\le r,&\\
&\;1+\frac{1}{\rho}(r-s)&\mbox{if}\quad &r\le s\le r+\rho,&\\
&\;0 &\mbox{if}\quad &s\ge r+\rho.&
\end{aligned}\right.
\end{equation}
By using $h_{\rho}$ in \eqref{weaksolution} and letting
$\rho\rightarrow 0$, we deduce
\begin{equation}\label{integralsolution}
c_{N}r^{N-k}(-|v^{\prime}|^{k-1}v^{\prime})=\lambda\int_{0}^{r}s^{N-1}\big(\frac{N+2}{N}+f(s)\big)|v|^{\frac{2}{N}+f(s)}\mathrm{e}^{\mu_{N}|v|^{\frac{N+2}{N}+f(s)}}\,\mathrm{d}s,\;\;\mbox{a.e
on}\;\; (0,1).
\end{equation}
It follows that $v$ is a  decreasing function  such that  $v\in C^{1}(0,1]$. In addition, from \eqref{integralsolution}
\begin{equation}\label{gradiente}
-v^{\prime}(r)=\left(\frac{\lambda}{c_{N}r^{N-k}}\int_{0}^{r}s^{N-1}\big(\frac{N+2}{N}+f(s)\big)|v|^{\frac{2}{N}+f(s)}\mathrm{e}^{\mu_{N}|v|^{\frac{N+2}{N}+f(s)}}\,\mathrm{d}s\right)^{\frac{1}{k}}.
\end{equation}
Further, from \eqref{r-estimate} 
\begin{equation}\nonumber
\begin{aligned}
|v(r)|^{f(r)}
\le \big(-\frac{N}{\mu_{N}}\ln r\big)^{\frac{N}{N+2}f(r)},\;\; \forall\, r\in (0,1).\\
\end{aligned}
\end{equation}
By \eqref{limsup}, we have  $|\ln r|^{f(r)}\to 1$ as $r\to 0^{+}$ and thus  $|v|^{f(r)}\le 2$ on $(0,\rho)$ for some $\rho>0$ small enough. Thus, from Lemma~\ref{0-beha} we obtain 
\begin{equation}\nonumber
\lim_{r\rightarrow 0^{+}}r^{k}|v|^{\frac{2}{N}+f(r)}\mathrm{e}^{\mu_{N}|v|^{\frac{N+2}{N}+f(r)}}=0.
\end{equation}
Hence, \eqref{gradiente} and L'Hospital's rule yield $\lim_{r\rightarrow 0^{+}}v^{\prime}(r)=0$ and $v\in C^{1}[0,1]$. Using \eqref{gradiente}, we get  $v\in C^{2}(0,1]$ and we can write
\begin{equation}\label{hessian}
v^{\prime\prime}(r)=\frac{v^{\prime}(r)}{k}\left[-\frac{(N-k)}{r}+\frac{r^{N-1}(\frac{N+2}{N}+f(r))|v|^{\frac{2}{N}+f(r)}\mathrm{e}^{\mu_{N}|v|^{\frac{N+2}{N}+f(r)}}}{\int_{0}^{r}s^{N-1}(\frac{N+2}{N}+f(s))|v|^{\frac{2}{N}+f(s)}\mathrm{e}^{\mu_{N}|v|^{\frac{N+2}{N}+f(s)}}\,\mathrm{d}s}\right].
\end{equation}
From \eqref{gradiente},  using  $f(0)=0$ and $v\in C^{1}[0,1]$ we can write
\begin{equation}\label{theta-limit}
\lim_{r\rightarrow 0^{+}} \frac{v^{\prime}(r)}{r}=-\left(\frac{\lambda}{N c_{N}}(\frac{N+2}{N})|v(0)|^{\frac{2}{N}}\mathrm{e}^{\mu_{N}|v(0)|^{\frac{N+2}{N}}}\right)^{\frac{2}{N}}<0.
\end{equation}
Note that  $v^{\prime}<0$ for $0<r\le 1$. In addition, by \eqref{gradiente}
\begin{equation}\nonumber
\begin{aligned}
\frac{v^{\prime}r^{N-1}(\frac{N+2}{N}+f(r))\mathrm{e}^{\mu_{N}|v|^{\frac{N+2}{N}+f(r)}}}{\int_{0}^{r}s^{N-1}(\frac{N+2}{N}+f(s))\mathrm{e}^{\mu_{N}|v|^{\frac{N+2}{N}+f(s)}}\,\mathrm{d}s}&= \frac{\lambda}{c_{N}}\frac{v^{\prime}r^{N-1}(\frac{N+2}{N}+f(r))\mathrm{e}^{\mu_{N}|v|^{\frac{N+2}{N}+f(r)}}}{r^{N-k}(-v^{\prime})^{k}} \\
&= -\frac{\lambda}{c_{N}}\left(-\frac{r}{v^{\prime}}\right)^{k-1}(\frac{N+2}{N}+f(r))\mathrm{e}^{\mu_{N}|v|^{\frac{N+2}{N}+f(r)}}
\end{aligned}
\end{equation}
which  (by \eqref{theta-limit}) converges as  $r\rightarrow 0^{+}$. Hence, from \eqref{hessian} there exists $\lim_{r\rightarrow 0^{+}}v^{\prime\prime}(r)$ and $v\in C^{2}[0,1]$.
\end{proof}
\end{lemma}
\section{Proof of Theorem~\ref{thm2}: k-admissible extremals}
In this section we use a maximizer $v\in X_{1}$ for  $\mathcal{MT}_{N}(f)$, obtained in the Section~\ref{sec3}, to build an extremal function $u\in\Phi^{k}_{0,\mathrm{rad}}(B)$ for  $ \mathrm{MT}_{N}(f)$. 
For that, consider $u: B\rightarrow\mathbb{R}$  given by
\begin{equation}\label{u0difinition}
u(x)=-v(r),\; r=|x| \;\; \mbox{for}\;\; x\in \overline{B}.
\end{equation}
\begin{lemma}\label{k-admi} Let $u: B\rightarrow\mathbb{R}$ defined by  \eqref{u0difinition}. Then  $u\in \Phi^{k}_{0,\mathrm{rad}}(B)$ is a maximizer for $ \mathrm{MT}_{N}(f)$.
\end{lemma}
\begin{proof}
We can assume that the maximizer $v\in X_{1}$ for $\mathcal{MT}_{N}(f)$ is a nonnegative function.  
Thus, from Lemma~\ref{vc2} we have $v\in C^{2}[0,1]$ and consequently $u\in C^{2}(B)$ with $u_{|\partial B}=0$.  
To conclude that  $u\in \Phi^{k}_{0,\mathrm{rad}}(B)$, we only need to show that $F_j[u]\ge 0$, for $j=1,2,\dots, k.$
From the radial form of the $k$-Hessian operator (see, \cite{Wei}), we have
$$
F_{j}[u]=\frac{1}{j}\binom{N-1}{j-1}r^{1-N}\left(r^{N-j}(u^{\prime})^{j}\right)^{\prime}.
$$
Thus,  it enough verify that 
$$
\left(r^{N-j}(u^{\prime})^{j}\right)^{\prime}\ge 0, \quad j=1,\dots, k,
$$
or equivalently
\begin{equation}\label{v-kadmi}
\left(r^{N-j}(-v^{\prime})^{j}\right)^{\prime}\ge 0, \quad j=1,\dots, k.
\end{equation}
Using \eqref{gradiente} and the assumption $N=2k$, it is easy to show that 
\begin{equation}\nonumber
\left(r^{N-j}(-v^{\prime})^{j}\right)^{\prime} =r^{N-2j}\left[\Phi (r)\right]^{\frac{j}{k}}\left[\frac{N-2j}{r}+\frac{j}{k}\frac{\Phi^{\prime}(r)}{\Phi(r)}\right],
\end{equation}
where 
$$
\Phi(r)=\frac{\lambda}{c_{N}}\int_{0}^{r}s^{N-1}\Big(\frac{N+2}{N}+f(s)\Big)|v|^{\frac{2}{N}+f(s)}\mathrm{e}^{\mu_{N}|v|^{\frac{N+2}{N}+f(s)}}\,\mathrm{d}s.
$$
Thus, since $\Phi^{\prime}(r)\ge 0$, for  $0<r\le 1$, we conclude that \eqref{v-kadmi} holds. Thus, $u\in \Phi^{k}_{0,\mathrm{rad}}(B)$. Moreover, from \eqref{NormX} and \eqref{FX2} we have $\|u\|_{\Phi^{k}_0}=\|v\|_{X_1}=1$ and 
\begin{equation}\nonumber
\begin{aligned}
\int_{B}\mathrm{e}^{\mu_{N}|u|^{\frac{N+2}{N}+f(|x|)}}\ud x&=\omega_{N-1}\int_{0}^{1} r^{N-1}\mathrm{e}^{\mu_{N}|v|^{\frac{N+2}{N}+f(r)}}\ud r\\
&=\omega_{N-1}\mathcal{MT}_{N}(f)\ge \mathrm{MT}_{N}(f).
\end{aligned}
\end{equation} 
Thus, 
\begin{equation}\nonumber
\begin{aligned}
\int_{B}\mathrm{e}^{\mu_{N}|u|^{\frac{N+2}{N}+f(|x|)}}\ud x= \mathrm{MT}_{N}(f),
\end{aligned}
\end{equation} 
which is the desired conclusion.
\end{proof}


\end{document}